\documentclass[10pt]{amsart}
\usepackage{amsmath}
\usepackage{amsthm}
\usepackage{amsfonts}
\usepackage{mathrsfs}
\usepackage{amssymb}
\usepackage{amscd}
\usepackage{pgf}
\usepackage{tikz}
\usetikzlibrary{cd}
\usepackage{amscd}

\usepackage{ytableau}

\expandafter\let\csname ver@amsthm.sty\endcsname\relax
\let\theoremstyle\relax

\usepackage[utf8]{inputenc}
\usepackage[
    breaklinks,
    colorlinks,
    citecolor=teal,
    linkcolor=teal,
    urlcolor=teal,
    pagebackref=true,
    hyperindex
    ]{hyperref}
\usepackage{fancyhdr}
\usepackage[
    hscale=0.7,
    vscale=0.75,
    headheight=13pt,
    centering,
    ]{geometry}
\usepackage{amsmath}
\usepackage{amsthm}
\usepackage{amssymb}
\usepackage{mathtools}
\usepackage{stmaryrd}
\usepackage{mathdots}
\usepackage[all,cmtip]{xy}
\usepackage{xcolor}
\usepackage{framed}
\usepackage{pgffor}
\usepackage[capitalize]{cleveref}

\theoremstyle{plain}

\newtheorem{theorem}{Theorem}[section]
\newtheorem{proposition}[theorem]{Proposition}
\newtheorem{lemma}[theorem]{Lemma}
\newtheorem{corollary}[theorem]{Corollary}

\newtheorem{conjecture}[theorem]{Conjecture}

\theoremstyle{definition}

\newtheorem{definition}[theorem]{Definition}

\newtheorem{remark}[theorem]{Remark}
\newtheorem{example}[theorem]{Example}
 
    \newtheorem{Exa}[theorem]{Example}
    \newtheorem{Rem}[theorem]{Remark}
    \newtheorem{Cons}[theorem]{Construction}
    
		\newtheorem{Question}[theorem]{Question}

\numberwithin{figure}{section}

\numberwithin{equation}{section}

\makeatletter

\def\@myMR[#1 #2]{\relax\ifhmode\unskip\spacefactor3000 \space\fi
  \MRhref{#1}{MR\,#1}}
\renewcommand\MR[1]{\@myMR[#1 ]}
\renewcommand{\MRhref}[2]{{\tiny%
  \href{http://www.ams.org/mathscinet-getitem?mr=#1}{#2}}%
}
\renewcommand*{\backref}[1]{}
\renewcommand*{\backrefalt}[4]{%
    \tiny%
    ({
    \ifcase #1 not cited%
          \or cit.\ on p.~#2%
          \else cit.\ on pp.~#2%
    \fi%
    })\\[-.6em]}

\def\maketitle{\par
  \@topnum\z@ 
  \@setcopyright
  \thispagestyle{empty}
  \ifx\@empty\shortauthors \let\shortauthors\shorttitle
  \else \andify\shortauthors
  \fi
  \@maketitle@hook
  \begingroup
  \@maketitle
  \toks@\@xp{\shortauthors}\@temptokena\@xp{\shorttitle}%
  \toks4{\def\\{ \ignorespaces}}
  \edef\@tempa{%
    \@nx\markboth{\the\toks4
      \@nx\MakeUppercase{\the\toks@}}{\the\@temptokena}}%
  \@tempa
  \endgroup
  \c@footnote\z@
    \renewcommand{\footnoterule}{%
      \kern -3pt
      \hrule width \textwidth height .5pt
      \kern 2pt
    }
  {
    \renewcommand\thefootnote{}
    \vspace{-2em}
    \footnote{
      \par\vspace{-1.2em}\noindent
      \def\@footnotetext##1{\noindent{\footnotesize##1}\par}%
      \let\@makefnmark\relax  \let\@thefnmark\relax
      \ifx\@empty\@date\else \@footnotetext{\@setdate}\fi
      \ifx\@empty\@subjclass\else \@footnotetext{\@setsubjclass}\fi
      \ifx\@empty\@keywords\else \@footnotetext{\@setkeywords}\fi
      \ifx\@empty\thankses\else \@footnotetext{%
        \def\par{\let\par\@par}\@setthanks}%
      \fi
    }
    \addtocounter{footnote}{-1}
  }
  \@cleartopmattertags
}

\def\@adminfootnotes{\@empty}

\def\@settitle{\begin{center}%
  \baselineskip14\p@\relax
    \bfseries
\Large
  \@title
  \end{center}%
}

\def\@setauthors{%
  \begingroup
  \def\thanks{\protect\thanks@warning}%
  \trivlist
  \centering\footnotesize \@topsep30\p@\relax
  \advance\@topsep by -\baselineskip
  \item\relax
  \author@andify\authors
  \def\\{\protect\linebreak}%
  \large{\authors}%
  \ifx\@empty\contribs
  \else
    ,\penalty-3 \space \@setcontribs
    \@closetoccontribs
  \fi
  \endtrivlist
  \endgroup
}

\def\@setaddresses{\par
  \nobreak \begingroup
\footnotesize
  \def\author##1{\end{minipage}\hskip 1sp \begin{minipage}{.5\textwidth}\raggedright%
    ~\\[2em]{\bf##1}\\[.5em]%
  }%
  \interlinepenalty\@M
  \def\address##1##2{\begingroup
    {\ignorespaces##2}\endgroup\\[.5em]}%
  \def\curraddr##1##2{\begingroup
    \@ifnotempty{##2}{\nobreak\indent\curraddrname
      \@ifnotempty{##1}{, \ignorespaces##1\unskip}\/:\space
      ##2\par}\endgroup}%
  \def\email##1##2{\begingroup
    \@ifnotempty{##2}{\nobreak\indent
      \@ifnotempty{##1}{, \ignorespaces##1\unskip}
      \ttfamily##2\par}\endgroup}%
  \def\urladdr##1##2{\begingroup
    \def~{\char`\~}%
    \@ifnotempty{##2}{\nobreak\indent\urladdrname
      \@ifnotempty{##1}{, \ignorespaces##1\unskip}\/:\space
      \ttfamily##2\par}\endgroup}%
  \setlength{\parindent}{0pt}%
  \vfill%
  {
  \begin{minipage}{0mm}
  \addresses
  \end{minipage}
  }
  \endgroup
}

\renewcommand{\author}[2][]{%
  \ifx\@empty\authors
    \gdef\authors{#2}%
    \g@addto@macro\addresses{\author{#2}}%
  \else
    \g@addto@macro\authors{\and#2}%
    \g@addto@macro\addresses{\author{#2}}%
  \fi
  \@ifnotempty{#1}{%
    \ifx\@empty\shortauthors
      \gdef\shortauthors{#1}%
    \else
      \g@addto@macro\shortauthors{\and#1}%
    \fi
  }%
}
\edef\author{\@nx\@dblarg
  \@xp\@nx\csname\string\author\endcsname}

\def\@secnumfont{\@empty}

\def\section{\@startsection{section}{1}%
  \z@{.7\linespacing\@plus\linespacing}{.5\linespacing}%
  {\large\bfseries\centering}}

\makeatother

\usepackage{enumerate}
\usepackage{hyperref}

\title{Parabolic Lusztig varieties and chromatic symmetric functions}
\author{Alex Abreu}

\address{
    Instituto de Matemática e Estatística\\
    Universidade Federal Fluminense\\
    Rua Prof. M. W. de Freitas, S/N\\
    24210-201 Niterói, Rio de Janeiro, Brasil
}

\email{alexbra1@gmail.com}

\author{Antonio Nigro}

\address{
    Instituto de Matemática e Estatística\\
    Universidade Federal Fluminense\\
    Rua Prof. M. W. de Freitas, S/N\\
    24210-201 Niterói, Rio de Janeiro, Brasil
}

\email{antonio.nigro@gmail.com}

\usepackage{color}
\definecolor{forestgreen}{rgb}{0.13, 0.55, 0.13}

\newcommand{\col}{\colon}

\newcommand{\X}{\mathcal X}

\newcommand{\F}{\mathcal F}
\newcommand{\Y}{\mathcal Y}
\newcommand{\U}{\mathcal U}
\newcommand{\N}{\mathcal N}

\newcommand{\mS}{\mathcal S}

\newcommand{\m}{\mathbf{m}}
\newcommand{\ind}{i}
\newcommand{\parti}{\mathcal{P}ar}

\DeclareMathOperator{\Adm}{Adm}

\DeclareMathOperator{\Ima}{Im}
\DeclareMathOperator{\ch}{ch}

\DeclareMathOperator{\LLT}{LLT}

\DeclareMathOperator{\csf}{csf}
\DeclareMathOperator{\asc}{asc}

\DeclareMathOperator{\Gr}{Gr}

\newcommand{\C}{\mathbb{C}}  
\newcommand{\flag}{\mathcal{B}}

\newcommand{\h}{\mathcal{Y}}  
\newcommand{\Hi}{\mathcal{H}}

\newcommand{\dw}{\dot{w}} 
 
\newcommand{\dz}{\dot{z}}

\newcommand{\yd}[1]{\scalebox{0.3}{\ydiagram{#1}}}

\begin{document}

\maketitle


   \begin{abstract}
    The characters of Kazhdan--Lusztig elements of the Hecke algebra over $S_n$ (and in particular, the chromatic symmetric function of indifference graphs) are completely encoded in the (intersection) cohomology of certain subvarieties of the flag variety. Considering the forgetful map to some partial flag variety, the decomposition theorem tells us that this cohomology splits as a sum of intersection cohomology groups with coefficients in some local systems of subvarieties of the partial flag variety. We prove that these local systems correspond to representations of subgroups of $S_n$. An explicit characterization of such  representations would provide a recursive formula for the computation of such characters/chromatic symmetric functions, which could settle Haiman's conjecture about the positivity of the monomial characters of Kazhdan--Lusztig elements and Stanley--Stembridge conjecture about $e$-positivity of chromatic symmetric function of indifference graphs. We also find a connection between the character of certain homology groups of subvarieties of the partial flag varieties and the Grojnowski--Haiman hybrid basis of the Hecke algebra.

\end{abstract}

\tableofcontents

\section{Introduction}

The connection between chromatic symmetric functions and geometry goes back to Stanley (\cite{Stanley86}), who noticed that, based on a recursion of Procesi  (\cite{Procesi}), the chromatic symmetric function of the path graph is the omega-dual of the Frobenius character of the cohomology of the toric variety given by the Weyl chambers of the symmetric group $S_n$.

The Shareshian--Wachs (\cite{ShareshianWachs}) conjecture, now proved by Brosnan--Chow (\cite{BrosnanChow}) and Guay-Paquet (\cite{GP}), generalizes this connection  to any indifference graph (unit interval order graph). Namely, the chromatic quasisymmetric function of an indifference graph is the omega-dual Frobenius character of the cohomology of an associated Hessenberg variety. This motivated a flurry of work trying to better understand the cohomology of Hessenberg varieties, including \cite{Precup18}, \cite{HaradaPrecup}, \cite{CHL}, \cite{HHMPT}, \cite{HPT21}, and \cite{BalibanuCrooks} to name a few. This connection can be used in both directions. For instance, some relations for the chromatic symmetric function, such as the modular law in \cite{GPmodular}, or \cite[Theorem 1.1]{AN}  can be used to obtain results about the geometry of these Hessenberg varieties (see \cite{PrecupSommers}, \cite{KiemLee}).

The chromatic symmetric function is also related to the characters of the Kazhdan--Lusztig basis elements $C'_w$ of the Hecke algebra of $W=S_n$. In fact, when $w$ is codominant the characters of $C'_w$ recover the chromatic quasisymmetric function of the associated indifference graph (see \cite{CHSS} and \cite{AN_haiman}). In particular, we have a relation between the characters of $C'_w$, with $w$ codominant, and the cohomology of Hessenberg varieties.

For general $w\in S_n$, the characters of $C'_w$ also have geometric meaning. Namely, they are encoded in the geometry of certain subvarieties of the flag variety $\flag := GL_n/B$, where $B$ is the Borel subgroup o upper triangular matrices. These varieties $\h_w(X)$, which we call \emph{Lusztig varieties}, are defined by
\[
\h_w(X) = \{V_\bullet=(0\subset V_1\subset V_2\subset \ldots \subset V_n=\mathbb{C}^n); \dim XV_i\cap V_j\geq r_{i,j}(w)\},
\]
where $w$ is a permutation in $S_n$, $X$ is an invertible matrix, and
\[
r_{i,j}(w): = |\{k;k\leq i,w(k)\leq j\}|.
\]
If $w$ is codominant, then the variety $\h_w(X)$ is a Hessenberg variety.

\begin{example}
If $w=3412 \in S_4$, then
\[
\h_w(X) = \{V_1\subset V_2\subset V_3\subset \mathbb{C}^4; XV_1\subset V_3, V_1\subset XV_3\}.
\]
Note that some redundant conditions $\dim XV_i\cap V_J\geq r_{i,j}(w)$ are omitted above.
\end{example}

When $X$ is regular semisimple (for example if $X$ is diagonal with distinct diagonal entries), the intersection cohomology group\footnote{The varieties $\h_w(X)$ are singular when the permutation $w$ is singular, so it is natural to consider the intersection cohomology instead of usual cohomology.} $IH^*(\h_w(X))$ is naturally a $S_n$-module, and its (graded) Frobenius character, denoted by $\ch(IH^*(\h_w(X)))$, is the same as
\[
\ch(q^{\frac{\ell(w)}{2}}C'_w):=\sum_{\lambda\vdash n}\chi^{\lambda}(q^{\frac{\ell(w)}{2}}C'_w)s_{\lambda},
\]
which is a result of Lusztig \cite{ChaShvV} (see also \cite{AN_hecke}, the notation of which we follow). Here, $\chi^{\lambda}$ is the (extension to the Hecke algebra of the) irreducible character of $S_n$ associated to the partition $\lambda$ and $s_{\lambda}$ is the Schur symmetric function associated to $\lambda$. The Stanley--Stembridge conjecture about $e$-positivity of chromatic symmetric function and the Haiman conjecture about $h$-positivity of characters of Kazhdan--Lusztig basis elements become equivalent to the fact that $IH^*(\h_w(X))$ has a permutation basis (stabilizers of which are conjugate to the Young subgroups $S_{\lambda}$ for some partition $\lambda$). \par

 One strategy to prove that such a permutation basis exists is to write $IH^*(\h_w(X))$ as sum of $S_n$-modules, where each one has a permutation basis. There are several natural ways to write $IH^*(\h_w(X))$ as a sum of $S_n$-modules using the decomposition theorem of \cite{BBD}, although proving that each one has a  permutation basis still appears to be a very difficult problem.\par

 Let $J$ be a subset of the set $S=\{1,\ldots, n-1\}$ of simple transpositions of $S_n$ and denote by $\flag_J = GL_n/P_J$ the partial flag variety associated to $J$, where $P_J$ is the parabolic subgroup of $GL_n$ induced by $J$. That is, if $S\setminus J=\{i_1,\ldots, i_k\}$, then $\flag_J=\{V_{i_1}\subset V_{i_2}\subset \ldots \subset V_{i_k}\subset \mathbb{C}^n\}$. By the decomposition theorem there exist subvarieties $\h_{\alpha}\subset \flag_J$ and local systems on (open sets of) $\h_{\alpha}$ such that
 \[
 IH^*(\h_w(X)) = \bigoplus_{\alpha} IH^*(\h_\alpha(X), L_\alpha).
 \]
Results of Lusztig \cite{LusztigParabolicI} give a complete characterization of the subvarieties $\h_\alpha$, which are indexed by elements of \footnote{This is in analogy to the Schubert varieties in the partial flag variety $\flag_J$, which are also indexed by elements in ${}^JS_n$, although the varieties $\h_{\alpha}(X)$ have a much more involved definition, see Section \ref{sec:parabolic} below.}
\[
{}^JS_n:=\{w\in S_n; w^{-1}(j)<w^{-1}(j+1)\text{ for all }j\in J\}.
\]
We write $\h_{z,J}(X)$ for the \emph{parabolic Lusztig variety} in $\flag_J$ associated to $z\in {}^JS_n$. The splitting of the intersection cohomology becomes
\begin{equation}
\label{eq:decomposition_hw}
IH^*(\h_w(X)) = \bigoplus_{z\in {}^JW} IH^*(\h_{z,J}(X), L_{z,w}^J).
\end{equation}
 We could try to repeat the same process to split the cohomology even further: If $w\in {}^JS_n$, $J'\supset J$, and $F$ a suitable local system on $\h_{w,J}(X)$, we write
\begin{equation}
\label{eq:decomposition_hzJ}
IH^*(\h_{w,J}(X),F)=\bigoplus_{z\in {}^{J'}S_n} IH^*(\h_{z,J'}(X), L_{z, w}^{J',J}(F)),
\end{equation}
and so on.

\begin{Exa}
\label{exa:path}
Consider the permutation $w=2341\in S_4$ and let $X$ be a regular semisimple matrix. Then $\h_w(X) = \{V_\bullet; XV_1\subset V_2, XV_2\subset V_3\}$. Let $J_1=\{3\}$ and apply the decomposition theorem to the map $f_1\col \h_w(X)\to \flag_{J_1}$, then we get
\[
IH^*(\h_w(X)) = IH^*(\h_{2341, J_1}(X))\oplus (IH^*(\h_{2314,J_1}(X))\otimes \mathbb{C}[-2]),
\]
where (for now these can be taken as definitions, but see Definition \ref{def:lusztig_parabolic} and Proposition \ref{prop:Gamma_J_injective})
\begin{align*}
\h_{2341,J_1}(X)=&\{V_1\subset V_2\subset \mathbb{C}^4;XV_1\subset V_2 \},\\
\h_{2134,J_1}(X)=&\{V_1\subset V_2\subset \mathbb{C}^4;XV_2\subset V_2 \}.
\end{align*}
Indeed, since we are forgetting $V_3$, we have that the image of $f_1$ is precisely $\h_{2341,J_1}$. Moreover, the fiber of $f_1$ over a flag $(V_1\subset V_2)\in \h_{2134,J_1}(X)$ is $\mathbb{P}^1$, while $f_1$ is an isomorphism over $\h_{2341,J_1}(X)\setminus \h_{2134,J_1}(X)$. \par

Let us proceed to the next step. Letting $J_2=\{2,3\}$ and applying the decomposition theorem to $g_1\col \h_{2341,J_1}(X)\to \flag_{J_2}$ and $g_2\col \h_{2134,J_2}(X)\to\flag_{J_2}$, we get
\begin{align*}
IH^*(\h_{2341,J_1}(X))=&IH^*(\h_{2341,J_2}(X))\oplus(IH^*(\h_{1234,J_2}(X))\otimes (\mathbb{C}[-2]\oplus\mathbb{C}[-4])),\\
IH^*(\h_{2134, J_1}(X))=&IH^*(\h_{2134,J_2}(X)),
\end{align*}
where
\begin{align*}
    \h_{2341,J_2}(X) =& \{V_1\subset \mathbb{C}^4\} = \mathbb{P}^3,\\
    \h_{1234,J_2}(X) = & \{V_1\subset \mathbb{C}^4; XV_1=V_1\} = \{(1:0:0:0), (0:1:0:0), (0:0:1:0), (0:0:0:1)\},\\
    \h_{2134,J_2}(X) = &\{V_1\subset \mathbb{C}^4; V_1\subset <e_i,e_j>\text{ for some }i,j\in \{1,2,3,4\}\} = \bigcup_{i,j}\mathbb{P}^1_{i,j}.
\end{align*}
Indeed, the image of $g_2$ is $\h_{2134,J_2}$, the map $\h_{2134,J_1}(X)\to \h_{2134,J_2}(X)$ is a normalization map, the image of $g_1$ is $\h_{2341,J_2}(X)$, the map $g_1$ is an isomorphism over $\h_{2341,J_2}(X)\setminus \h_{1234,J_2}(X)$, and the fiber of $g_1$ over $\h_{1234,J_2}(X)$ is $\mathbb{P}^2$. The graded Frobenius characters satisfy
\[
\ch(\h_{w}(X))=\ch(IH^*(\h_{2341,J_2}(X)))+  (q+q^2)\ch (IH^*(\h_{1234, J_2}(X))) + q\ch(IH^*(\h_{2134,J_2}(X))).
\]
Since everything is in $\mathbb{P}^3$ (which carries a natural action of $S_4$), we can compute
\begin{align*}
    \ch(IH^*(\h_{2341,J_2}(X)))=&(1+q+q^2+q^3)h_4,\\
    \ch(IH^*(\h_{1234,J_2}(X)))=& h_{3,1},\\
    \ch(IH^*(\h_{2134,J_2}(X))) = & (1+q)h_{2,2},
\end{align*}
and arrive at the well-known expression (see \cite[Table 1]{Haiman})
\[
\ch(IH^*(\h_{w}(X)))= (q+q^2)h_{2,2}+(q+q^2)h_{3,1}+(1+q+q^2+q^3)h_4,
\]
which is the $\omega$-dual of the chromatic quasisymmetric function of the path graph with $4$ vertices.
\end{Exa}

Our first result is a certain characterization of the local systems appearing in Equation \eqref{eq:decomposition_hw}. More explicitly, we prove that every local system $L_{z,w}^J$ appearing in Equation \eqref{eq:decomposition_hw} is induced by a representation of a subgroup of $S_n$. Given a permutation $z\in {}^JS_n$, define $J_z:= \bigcap_{n\in \mathbb{Z}} z^{n}Jz^{-n} $, in particular $zJ_zz^{-1} = J_z$, and set $(S_n)_{J_z}^z:=\{w\in (S_n)_{J_z}; zw=wz\}$, where $(S_n)_{J_z}$ is the subgroup of $S_n$ generated by the simple transpositions in $J_z$.  We will prove that $(S_n)_{J_z}^z$ is isomorphic to a product of symmetric groups (Proposition \ref{prop:WJw_product}).\par

To better state our results, it is preferable to work with families and to consider the perverse sheaves point of view. Define $\h_{w}\subset GL_n\times \flag$ and $\h_{z,J}\subset GL_n\times \flag_J$ to be the subvarieties whose fibers over a point $X\in GL_n$ are precisely $\h_w(X)$ and $\h_{z,J}(X)$, and let $f\col GL_n\times \flag \to GL_n\times \flag_J$ be the forgetful map. Then we can write a relative version of Equation \eqref{eq:decomposition_hw}:

\begin{equation}
\label{eq:decomposition_hw_rel}
f_*(IC_{\h_w}) = \bigoplus_{z\in {}^JW} IC_{\h_{z,J}}(L_{z,w}^J).
\end{equation}

\begin{theorem}
\label{thm:UwJ_intro}
There exists an open set $\U_{z,J}\subset \h_{z,J}$ the fundamental group of which has a natural map $\pi_1(U_{z,J},(X,V_\bullet))\to (S_n)_{J_z}^z$. Moreover, every local system $L_{z,w}^{J}$ appearing in Equation \eqref{eq:decomposition_hw_rel} is induced by a representation of $(S_n)_{J_z}^z$.
\end{theorem}

The proof of Theorem \ref{thm:UwJ_intro} is done in Sections \ref{sec:monodromy} and \ref{sec:classification}.\par

One possible extension of the Stanley--Stembridge conjecture/Haiman's conjecture is the following.
\begin{conjecture}
The local systems $L_{z,w}^J$ of Equation \eqref{eq:decomposition_hw} are induced by permutations representations of $(S_n)_{J_z}^z$.
\end{conjecture}
When $J=\{1,\ldots, n-1\}$ the above conjecture is equivalent to Haiman's conjecture about $h$-positivity of $\ch(q^{\frac{\ell(w)}{2}}C'_w)$. In this case, we have that ${}^{J}S_n=\{e\}$ and $(S_n)_{J_e}^e=S_n$. Moreover, $\h_{e,J}=GL_n$. A potentially stronger conjecture still would ask for the local systems $L_{z,w}^{J',J}(F)$ in Equation \eqref{eq:decomposition_hzJ} to be induced by permutation representations whenever $F$ is induced by a permutation representation. Unfortunately, there are examples where $\ch(IH^*(\h_{z},L_{z,w}^J))$ is not $h$-positive.

\begin{Exa}
Consider $w=3412$ and $J=\{1,3\}$. Let $f\col \h_{w}(X)\to \Gr(2,4)=\h_{w,J}(X)$ be the natural map. The decomposition theorem implies (see Example \ref{exa:ch_G24} for explicit computations)
\[
IH^*(\h_{3412})=IH^*(\Gr(2,4), L)\oplus IH^*(\h_{1342,J})[-2] \oplus IH^*(\h_{3142,J}(X))[-2]\oplus IH^*(\h_{1234,J})[-4],
\]
where $L$ is the local system induced by the action of $S_2 = (S_4)_{J}^w$ on itself. Then we can compute the Frobenius character of each part and arrive at
\[
\ch(IH^*(\Gr(2,4), L)) = (q+q^2+q^3)h_{2,2}-(q+q^3)h_{3,1}+(1+2q+2q^2+2q^3+q^4)h_4.
\]
\end{Exa}
\medskip

One case where the group $(S_n)_{J_z}^z$ is easier to understand is when $J=\{n- k +1,\ldots, n-1\}$ for some positive integer $k$.  In this case, we can prove stronger results. This is done in  Section \ref{sec:J}. For instance, for each $z\in {}^JS_n$ we have that $J_z=\{n - k'+1,\ldots, n-1\}$ for some integer $k'=k'_J(z)\leq k$ and $(S_n)_{J_z}^z = (S_n)_{J_z} = S_1^{n-k'}\times S_{k'}$. Note that $k'$ can also be defined as the largest integer less than or equal to $k$ such that $z\in S_{n-k'}\times S_{1}^{k'}$. This means that whenever $k'>0$, $\h_{z,J}(X)$ is reducible and can be described as a union of irreducible components isomorphic to $\h_{\overline{z},\overline{J}}(\overline{X})\subset GL_{n-k'}/B_{n-k'}$ where $\overline{z}\in S_{n-k'}$ and $\overline{J} = \{n-k+1,\ldots, n-k'-1\}$.

For $z\in S_{n-k'}\times S_{1}^{k'}$ let $\overline{z}$ be the permutation in $S_{n-k'}$ that corresponds to $z$.

\begin{theorem}
\label{thm:chJnk1}
Let $J=\{n-k+1,\ldots, n-1\}$ be a subset of the set of simple transpositions of $S_n$ and let $w\in S_n$ be a permutation. Then we have that
\[
IH^*(\h_w)=\bigoplus_{z\in {}^JS_n}IH^*(\h_{z, J}, L_{z,w}),
\]
where $L_{z,w}$ is a local system on (an open set of) $\h_{z,J}$ induced by a representation $\rho_{z,w}$ of $S_{k'_J(z)}$. Moreover,
\[
\ch(IH^*(\h_{z,J}(X), L_{z,w})) =\ch(IH^*(\h_{\overline{z}, \overline{J}}(\overline{X})))\ch(\rho_{z,w}).
\]
\end{theorem}
 Changing $\h_w$ for a parabolic Lusztig variety $\h_{w,J}$ in the theorem above we get the immediate corollary.
\begin{corollary}
Let $J=\{n-k+1,\ldots, n-1\}$ and $J'=\{n-k'+1,\ldots, n-1\}$ be subsets of simple transpositions of $S_n$ with $k'>k$ and $w\in {}^JS_n$ a permutation. Also, let $f_{J,J'}\col GL_n\times \flag_J\to  GL_n\times \flag_{J'}$ be the forgetful map. Then
\[
(f_{J,J'})_*(IC_{\h_{w,J}})=\bigoplus_{z\in {}^{J'}S_n}IC_{\h_{z, J'}}(L_{z,w}^{J',J}),
\]
where $L_{z,w}^{J',J}$ is a local system on (a open set of) $\h_{z,J'}$ induced by a representation $\rho_{z,w}^{J',J}$ of $S_{k'_{J'}(z)}$. Moreover,
\[
\ch(IH^*(\h_{z,J'}(X), L_{z,w}^{J',J})) =\ch(IH^*(\h_{\overline{z}, \overline{J}}(\overline{X})))\ch(\rho_{z,w}^{J',J}).
\]
\end{corollary}

Understanding the local systems $L_{z,w}^{J',J}$  when $k'=k+1$ is sufficient to completely characterize $\ch(\h_w(X))$.  We make the following conjecture, which implies Haiman's $h$-positivity conjecture as well as Stanley--Stembridge conjecture.

\begin{conjecture}
Let $J=\{n-k+1,\ldots, n-1\}$ and $J'=\{n-k,\ldots, n-1\}$. Consider $w\in {}^JS_n$ an irreducible permutation (a permutation that is not contained in any proper Young subgroup of $S_n$) and write
\[
IH^*(\h_{w,J})=\bigoplus_{z\in {}^{J'}S_n} IH^*(\h_{z, J'}, L_{z,w}^{J',J}).
\]
Then $L_{z,w}^{J',J}$ is induced by a permutation representation of $(S_n)_{J'_z}$.
\end{conjecture}

When $J=\{2,\ldots, n-1\}$ and $X$ is a regular semisimple diagonal matrix, it is easier to characterize the varieties $\h_{z,J}(X)$. We have the following corollary.

\begin{corollary}
\label{cor:chHi}
Let $\Hi_i\subset \mathbb{P}^{n-1}$ be the union of all coordinate planes of codimension $i$. Then $IH(\h_w(X)) = \bigoplus_{i=0}^{n-1} IH^*(\Hi_i, L_{i,w})$, where $L_{i,w}$ is a local system on $\Hi_i$ induced by a representation $\rho_i$ of $S_i$. Moreover, $\ch(IH^*(\Hi_i,L_{i,w}))=[n-i]_qh_{n-i}\ch(\rho_i)$. In particular, if $L_{i,w}$ is induced by a permutation representation of $S_i$, then the same holds for $IH^*(\Hi_i,L_i)$.
\end{corollary}
In a subsequent work we will give a combinatorial description of $\ch(\rho_i)$ when $w$ is a codominant permutation.

Our final result concerns the character of the cohomology with compact support of the open cell of the parabolic Lusztig varieties, in Section \ref{sec:open}. In the non parabolic case, if we change the closed variety $\h_{w}(X)$ for its open cell $\h_{w}(X)^\circ:= \{V_\bullet; \dim XV_i\cap V_j= r_{i,j}(w)\}$, and the intersection cohomology for the cohomology with compact support, it is known that  $\ch(H^*_c(\h_w(X)^\circ) = \ch(T_w)$ where $T_w$ is the usual basis of the Hecke algebra. We can prove a similar result for the open cells $\h_{w,J}(X)^\circ$. Surprisingly, the character $\ch(H^*_c(\h_{w,J}(X)^\circ,L))$ agrees (up to multiplication by a polynomial in $q$) with the character of an element of the \emph{hybrid basis} $\{C'_{z}T_w\}_{z\in (S_n)_J, w \in {}^JS_n}$of the Hecke algebra introduced by Grojnowski--Haiman in \cite{HaimanGronj}  in order to prove the Schur positivity of LLT polynomials.

\begin{theorem}
\label{thm:groj_haiman}
Let $J$ and $w\in {}^J(S_n)$ be such that $wJw^{-1}=J$. Moreover, let $J'\subset J$ such that $wJ'w^{-1}=J'$ and let $L_J'$ be the local system on $\U_{w,J}\subset \h_{w,J}$ induced by the induced representation $\ind_{(S_n)_{J'}^w}^{(S_n)_J^w}$. Then
   \[
   \ch(IH^*(Y_{w,J}^\circ(X), IC_{\h_{w,J}^\circ(X)}(L_{J'})))=\ch(q^{\ell(w_{J'})}C'_{w_{J'}}T_w)/|(S_n)_{J'}|,
   \]
   where $w_{J'}$ is the maximum element of $(S_n)_{J'}$ and $|(S_n)_{J'}|:=\sum_{z\in (S_n)_{J'}}q^{\ell(z)}$.

\end{theorem}
We note that the condition $wJw^{-1}=J$ is not restrictive, as $\h_{w,J}(X)^\circ$ is isomorphic to $\h_{w,J_w}(X)^\circ$ (see \ref{prop:t_iso_t1}). Moreover, specializing to $q=1$ we can give a more explicit formula in terms of plethysm. Write $(S_n)_{J'}=S_{\lambda_1}\times S_{\lambda_2}\times \ldots \times S_{\lambda_\ell(\lambda)}$. We have that $w$ acts on $\{1,\ldots, \ell(\lambda)\}$ as a permutation $\sigma\in S_{\ell(\lambda)}$ and $\lambda_j=\lambda_{\sigma(j)}$.

\begin{theorem}
\label{thm:plethysm}
Write $\tau_1\ldots \tau_m$ for the cycle decomposition of $\sigma$,  then
\[
(\ch(C'_{w_{J'}}T_w)/|(S_n)_{J'}|)|_{q=1}=\prod_{j=1}^mp_{|\tau_j|}[h_{\lambda_{\tau_j}}].
\]
\end{theorem}

Finally, we can ask for combinatorial interpretations of $\ch(\h_{w,J}(X))$. When $w'$ is a codominant permutation that is maximal in $W_Jw'W_J$ and $w$ is the minimum element in $W_Jw'$, we find such a combinatorial interpretation of $\ch(\h_{w,J}(X))$, which is given by a modification, introduced by Gasharov, of the chromatic (quasi)symmetric function in terms of multicolorings, see Example \ref{exa:chromatic_gasharov}.  In section \ref{sec:further_directions} we collect some more questions about parabolic Lusztig varieties.

\section{Preliminaires}
In this section we review some results and fix basic notation for algebraic groups, perverse sheaves, and character sheaves, following \cite[Sections 2, 3 and 4]{AN_hecke}.

\subsection{Algebraic groups}
 We write $G$ for a connected reductive algebraic group, $T$ a maximal torus of $G$, $B$ a Borel subgroup of $G$, and $U$ the unipotent subgroup such that $B=TU$. Let $B^-$ and $U^-$ be the opposed groups of $B$ and $U$ respectively. Let $W$ be the Weyl group of $G$, and for each $w\in W$ we write $\dw$ for a representative of $w$ in $G$. We also denote by $S$ the set of simple transpositions of $W$.

  We denote by $\Delta$ (respectively, $\Phi$, respectively, $\Phi^+$) the set of simple roots (respectively, roots, respectively, positive roots) of $G$ with respect to $T\subset B$. For $J\subset S$, let $P_J$ denote the induced parabolic subgroup of $G$, $L_J$ the induced Levi subgroup, and $\Delta_J$ (respectively, $\Phi_J$) the set of simple roots (respectively, roots) of $L_J$ with respect with $T\subset B_J:=L_J\cap B$. We also let $U^J$ be the unipotent subgroup such that $P_J=L_JU^J$. For each $w\in W$, we define $U^w := U\cap \dw U^-\dw^{-1}$, $U_w:=U\cap \dw U\dw^{-1}$,  and given $J\subset S$, set $J_w:= \bigcap_n w^{n}Jw^{-n} $. If $J$ is such that $wJw^{-1}=J$, then conjugation by $w$ induces an automorphism of $W_J$. We denote by $W_J^w$ the fixed points of this automorphism, $W_J^w=\{z\in W_J; wz=zw\}$.

\subsection{Hecke Algebras in type A}

The Hecke algebra $H_n$ of the symmetric group $S_n$ is a $q$-deformation of the group algebra $\mathbb{C}[S_n]$. It is defined as the algebra over $\mathbb{C}(q^{\pm\frac{1}{2}})$ generated by $T_{s_1},\ldots, T_{s_{n-1}}$ such that
\begin{align*}
        T_{s_i}^2=&(q-1)T_{s_i}+q\\
        T_{s_{i}}T_{s_{i+1}}T_{s_i}=&T_{s_{i+1}}T_{s_i}T_{s_{i+1}}\\
        T_{s_i}T_{s_j}=&T_{s_j}T_{s_i}\quad \quad\quad\text{ if } |i-j|>1.
    \end{align*}
If $w=\prod_{j=1}^{\ell(w)} s_{i_j}$ is a reduced expression for $w\in S_n$, we define $T_w=\prod_{j=1}^{\ell(w)}T_{s_{i_j}}$. These $T_w$ form a basis for $H_n$ as a $\mathbb{C}(q^{\pm\frac{1}{2}})$ vector space, which has an involution $\iota\col H_n\to H_n$ given by  $\iota(T_w)=T_{w^{-1}}^{-1}$ and $\iota(q^{\frac{1}{2}})=q^{-\frac{1}{2}}$. The Kazhdan--Lusztig basis $\{C'_w\}_{w\in S_n}$ of $H_n$ is defined by the following properties:
\begin{enumerate}
    \item $\iota(C'_w)=C'_w$,
    \item there exists polynomials $P_{z,w}(q)$, for each $z,w\in S_n$, such that:
\begin{enumerate}
    \item the degree of $P_{z,w}(q)$ is at most $\frac{\ell(w)-\ell(z)-1}{2}$ if $z< w$ (in the Bruhat order),
    \item if $z=w$ we have that $P_{w,w}(q)=1$,
    \item if $z\not \leq w$, then $P_{z,w}(q)=0$,
    \item we have an equality $q^{\frac{\ell(w)}{2}}C'_w=\sum_{z\leq w} P_{z,w}(q)T_z$.
\end{enumerate}
\end{enumerate}
    The polynomials $P_{z,w}$ are the so-called Kazhdan--Lusztig polynomials.

\subsection{Perverse sheaves}
In this brief section we collect some results about perverse sheaves and small maps needed in the sequel, first fixing some notation. For a variety $X$, denote by $\mathbb{C}_X$ the trivial sheaf on $X$ with stalk $\mathbb{C}$. If $L$ is a local system on a open set $U$ of $X$, denote by $IC_X(L)$ the intermediate extension of $L$ to $X$, also called the intersection cohomology sheaf with coefficients on $L$. We begin with the following result about summands of perverse sheaves and smooth maps.

\begin{proposition}[{\cite[4.2.5 and 4.2.6]{BBD}}]
 \label{prop:smooth_pullback_perverse}
 If $f\col \X\to \mS$ is a smooth map of relative dimension $d$ and $\F$ is a simple perverse sheaf on $\mS$ then $f^*(\F)[d]$ is a simple perverse sheaf on $\X$. In particular, a simple perverse sheaf $\F$ on $\mS$ is a summand of a complex $\F'$ if and only if $f^*(\F)[d]$ is a summand of $f^*(\F')$.
 \end{proposition}

Now, we  turn our attention to small maps. A proper morphism $f\col X\to Y$ with $X$ smooth  is called  \emph{small} if, for every $i\geq 0$, we have that
\[
\dim \big( \{y\in Y; \dim(f^{-1}(y))\geq i\}\big ) < \dim(Y) -2i.
\]
We let $U$ be a open set of $Y$ where the restriction  $g:=f|_{f^{-1}(U)}\col f^{-1}(U)\to U$ is smooth, in particular, $g$ is an unramified covering. We let $L$ be the local system $g_*(\mathbb{C}_{f^{-1}(U)})$. We have the following result due to Borho--Macpherson.

\begin{proposition}[{\cite{BorhoMacpherson}}]
\label{prop:small_local_system}
  If $f\col X\to Y$ is small, and $U\subset Y$ is an open subset such that $g:=f|_{f^{-1}(U)}\col f^{-1}(U)\to U$ is smooth,  then $f_*(\mathbb{C}_X)=IC_Y(L)$.
\end{proposition}

An important example of a small map is the natural map $G\times^BB\to G$. As usual, $G\times^BB$ is the quotient $\frac{G\times B}{B}$, where $B$ acts on $G\times B$ as $b\cdot(g, b_0)=(gb^{-1},bb_0b^{-1})$. And the map is the projection onto the first factor. Lusztig proved:

\begin{proposition}[{\cite[Proposition 4.5]{LusztigIntersectionCohomology}}]
\label{prop:Grothe_Springer_Small}
    The natural map $\frac{G\times B}{B}\to G$ is small.
\end{proposition}

\subsection{Character Sheaves}
\label{sec:chashv}
   Every element $g$ of a Bruhat stratum $ BwB$ can be written uniquely as $u\dw b$, with $u\in U_w$ and $b\in B$. Since $B=TU$, we have a projection $p\col B\to T$, which induces a projection $p_w\col BwB\to T$ given by $p_w(u\dw b)=p(b)$.   The map $p_w$ is $B$-equivariant, where $B$ acts on $B\dw B$ by conjugation and $B$ acts on $T$ as $tu\cdot t' = t't\dw t^{-1}\dw^{-1}$. If $L$ is a local system on $T$, we can pull it back via $p_w$ to obtain a local system on $BwB$.  Note that $p_w^*(L)$ is $B$-invariant. We can identify $\h_w^\circ:=\{(X,gB); g^{-1}Xg\in B\dw B \}$ with $\frac{G\times B\dw B}{B}$ (where $B$ acts on $BwB$ by conjugation and on $G$ by $b\cdot g = gb^{-1}$) via the isomorphism
\begin{align*}
    \h_w^\circ &\to \frac{G\times B\dw B}{B}\\
    (X,gB)& \to (g, g^{-1}Xg).
\end{align*}
 Since $p_w^*(L)$ is $B$-invariant, $p_w^*(L)$ will induce a local system $L_w$ on $\h_w^\circ$. We let $f\col \h_w^\circ \to G$ be the projection onto the first factor, that is $f(X,gB)=X$. \par

   When $L$ is good enough, that is, when there exists a positive integer $n$ such that $L^{\otimes n}$ is the trivial sheaf on $T$ (these are called Kummer local systems), Lusztig proved that $f_!(L_w)$ is a semisimple complex, that is, it is a sum of shifted simple perverse sheaves. Each simple perverse sheaf that is a summand of $f_!(L_w)$ for some $w$ and some $L$ is called a \emph{character sheaf}.\par

   The study of character sheaves, which includes their classification, was done by Lusztig in \cite{ChaShvI,ChaShvII, ChaShvIV,ChaShvV}. See also \cite{MarsSpringer}. As an example, when $w=e$ is the identity and $L$ is the trivial sheaf, then the simple summands appearing in $f_!(L_w)$ correspond to the irreducible representations of $W$.  Indeed, there is a map $\pi_1(G^{rs},y)\to W$ (see \cite[Section 8.5]{BrosnanChow}, \cite[Section 3.1]{AN_hecke}), and each irreducible representation $\rho$ of $W$ will induce an irreducible local system $L_\rho$ on $G^{rs}$, with its intermediate extension $IC_G(L_\rho)$ an irreducible character sheaf. Moreover, $f_!(L_e)$ is induced by the regular representation of $W$. In the case that $G=GL_n$, these are the only character sheaves induced by the trivial sheaf on $T$.\par

   As noted by Lusztig, we could also work with the closure $\h_w=\overline{\h_w^\circ}$, changing $f_!(L_w)$ to $f_*(IC_{\h_w}(L_w))$ and define character sheaves as the simple summands of $f_*(IC_{\h_w}(L_w))$.


   In this work we are mostly concerned with the case that $L$ is trivial, and then $L_w$ is trivial as well. We call the character sheaves induced by the trivial local system \emph{$1$-character sheaves}, as they satisfy that $L^{\otimes 1}$ is trivial. 

  We collect some of Lusztig's results for the case $G=GL_n$ in the following proposition. If $\underline{w}=(w_1,\ldots, w_m)$ is a sequence of elements of $W$ we define
  \begin{equation}
  \label{eq:Yw}
  \h_{\underline{w}}^\circ = \{(X,g_0B,\ldots, g_mB); g_0^{-1}Xg_m^{-1}\in B, g_{i+1}^{-1}g_i\in B\dw_{i+1} B\}
  \end{equation}
  and write $f_{\underline{w}}\col \h_{\underline{w}}^{\circ}\to G$ be the projection on the first factor. We also define $\h_{\underline{w}}$ to be the closure of $\h_{\underline{w}}^\circ$ and set $\overline{f}_{\underline{w}}\col \h_{\underline{w}}\to G$ to be the projection on the first factor.
      \begin{proposition}
     \label{prop:char_sheaves_GLn}
       Let $G=GL_n$, with the notation above we have
       \begin{enumerate}
           \item The complex $(f_{\underline{w}})_!(\mathbb{C}_{\h_{\underline{w}}^{\circ}})$ is semisimple and every simple summand of $(f_{\underline{w}})_!(\mathbb{C}_{\h_{\underline{w}}^{\circ}})$ is a $1$-character sheaf. The complex $(f_{\underline{w}})_*(\mathbb{C}_{\h_{\underline{w}}})$ is semisimple and every simple summand of $(f_{\underline{w}})_*(\mathbb{C}_{\h_{\underline{w}}})$ is a $1$-character sheaf.
           \item Every $1$-character sheaf is a summand of $f_*(IC_{\h_e})$.
           \item The $1$-character sheaves are in bijection with the irreducible representations of $S_n$ and are precisely the intermediate extensions of the local systems on $G^{rs}$ induced by such representations via the map $\pi_1(G^{rs},X)\to S_n$.
       \end{enumerate}

     \end{proposition}

 If a complex $F=\bigoplus F_i[n_i]$ on $G=GL_n$ is such that each $F_i$ is the perverse sheaf induced by an irreducible representation $\rho_{\lambda_i}$ of $S_n$,  then we write
 \[
 \ch(F):=\sum q^{\frac{-n_i}{2}}s_{\lambda_i}(x).
 \]

We also have the following characterizations.

\begin{proposition}[{\cite{ChaShvIV, ChaShvV, AN_hecke}}]
\label{prop:ch_lusztig}
    The following equalities hold
    \begin{enumerate}
        \item $\ch( (f_{w})_!(\mathbb{C}_{\h_w^\circ})) = \ch(H^*_c(\h_w^\circ(X)))= \ch(T_w)$.
        \item $\ch((f_{w})_*(\mathbb{C}_{\h_w^\circ})) = \ch(H^*(\h_w^\circ(X))) =  \ch(T_{w^{-1}}^{-1})$.
        \item $\ch( (\overline{f}_w)_*(IC_{\h_w}))= \ch(IH^*(\h_w(X))) = \ch (q^{\frac{\ell(w)}{2}}C'_w)$.

    \end{enumerate}
\end{proposition}

We can also define character sheaves when the algebraic group $\widehat{G}$ is disconnected, see \cite[Section 4.5]{LusztigParabolicI}. Let $G^0$ be the identity component and $G^1$ be a distinguished component. Also, fix an element $g_1\in G^1$ such that $g_1\widehat{B}g_1^{-1}=\widehat{B}$. As before, for $z\in \widehat{W}$, we define
\[
\h_{z}^\circ = \{(X,g\widehat{B}); X\in G^1, g\in G^0, g_1^{-1}g^{-1}Xg\in \widehat{B}\dz \widehat{B}\}.
\]
Clearly $\h_x^\circ$ does not depend on the choice of $g_1\in g_1\widehat{B}$. Also, the condition $g_1^{-1}g^{-1}Xg\in \widehat{B}\dz \widehat{B}$ is equivalent to the condition that the Borel subgroups $X(g\widehat{B}g^{-1})X^{-1}$, $g\widehat{B}g^{-1}$ of $G^0$ have relative position $z$. The $1$-character sheaves of $G^1$ are the simple perverse sheaves that are summands of $(f_z)_!(\mathbb{C}_{\h_z^\circ})$ for some $z\in W$.
  \begin{example}
  \label{exa:character_sheaves_LJ}
  Consider $G$ a connected reductive linear group with Weyl group $W$. Let $J$ be a subset of the set of simple transpositions of $W$ and let $w\in {}^JW$ be such that $wJw^{-1}=J$. Consider the disconnected linear group $\widehat{G}=N_G(L_J)$. In this case we have that
  $\widehat{W}=W_J$, $G^0=L_J$, and $\widehat{B}=B_J$. Take $G^1=\dw L_J$ and note that $\dw B_J\dw^{-1}=B_J$ by the condition $wJw^{-1}=J$. For $z\in \widehat{W}=W_J$, we have that
  \[
  \h_{z}^\circ=\{(\dw l, l_0B_J); l,l_0\in L_J, \dw^{-1}l_0^{-1}\dw l l_0\in B_J\dz B_J\}.
  \]
  \end{example}

\section{Parabolic varieties}
\label{sec:parabolic}
In this section we introduce the parabolic Lusztig varieties. We saw in the introduction that for a permutation $w$ and a matrix $X$, we can define
\[
\h_{w}(X) = \{V_\bullet; XV_i\cap V_j\geq r_{ij}(w)\}
\]
by analogy with the Schubert variety
\[
\Omega_{w, F_\bullet}=\{V_\bullet; F_i\cap V_j\geq r_{ij}(w)\}.
\]
The parabolic Schubert varieties are also defined in terms of a fixed flag $F_\bullet=F_1\subset \ldots \subset F_n=\mathbb{C}^n$. Given $J\subset \{1,\ldots, n-1\}$ a subset of the set of simple transpositions of $S_n$ and $w$ a permutation in ${}^JS_n$, we define
\[
\Omega_{w,J,F_\bullet} = \{V_\bullet; \dim F_i\cap V_j\geq r_{i,j}(w), j \notin J\}.
\]
The flag $V_\bullet$ above is a partial flag
\[
0=V_0\subset V_{i_1}\subset\cdots\subset V_{i_m}
\]
where $\{i_1,\ldots, i_m\}=J^c$, thereby comparing how the partial flag $V_\bullet$ intersects the flag $F_\bullet$. If we try to extend this idea to construct parabolic Lusztig varieties by simply comparing $XV_\bullet$ and $V_\bullet$, we will get fewer varieties than needed.

\begin{Exa}
 Consider $\Gr(1,n)=\mathbb{P}^{n-1}$. The only possible relative position of $XV_1$ and $V_1$ are
 \begin{enumerate}
     \item $\{V_1; \dim XV_1\cap V_1\geq 1\}$, which is the set of the points $(0:\ldots:0:1:0:\ldots:0)$.
     \item $\{V_1;\dim XV_1\cap V_1\geq 0\}$, which is the whole $\mathbb{P}^{n-1}$.
 \end{enumerate}
 However, as we saw in Example \ref{exa:path}, there are other varieties that appear in the decomposition theorem, namely the varieties $H_i$, which are the union of the coordinates planes of codimension $i$.
\end{Exa}

The definition of a parabolic Lusztig variety is given in Definition \ref{def:lusztig_parabolic} below. The main idea is to construct a flag $V'_\bullet$ from $V_\bullet$ and $X$ (Construction \ref{cons:lusztig_parabolic_flag}), compare the flags $XV'_\bullet$ and $V_\bullet$, and iterate this construction until it stabilizes.\par

Let us begin with the general construction for a linear algebraic group (following \cite{LusztigParabolicI}). Let $G$ be an linear algebraic group and $W$ its Weyl group with set of simple reflections $S$. An infinite sequence of pairs $(J_n,w_n)_{n\geq0}$, with $J_n\subset S$ and $w_n\in W$ is called \emph{admissible} if the following  conditions are satisfied for every $n\geq 0$
 \begin{enumerate}
     \item $J_{n+1} = J_{n} \cap w_n J_{n}w_n^{-1} $,
     \item $w_{n} \in {}^{J_n}W^{J_n}$,
     \item $w_{n+1} \in W_{J_n}w_{n}W_{J_n}$.
     \end{enumerate}
Here, as usual, we write $W^{J}$, ${}^JW$, and ${}^JW^J$ for the subsets of elements $w$ of $W$ such that $w$ is minimal in $wW_J$, $W_Jw$, and $W_JwW_J$, respectively.

 \begin{lemma}
  Let $(J_n,w_n)_{n\geq 0}$ be an admissible sequence of pairs. If $J_n=w_nJ_nw_n^{-1}$ for some $n\geq 0$, then $(J_{n+1},w_{n+1})=(J_n,w_n)$.
 \end{lemma}
 \begin{proof}
    By item (1) we have $J_{n+1}=J_n$, so that items (2) and (3) imply
    \[
    w_{n+1} \in {}^{J_n}W^{J_n} \cap W_{J_n}w_nW_{J_n}=\{w_n\}.
    \]
 \end{proof}

 Since $J_{n+1}\subset J_n$, there exists $n_0$ such that $J_{n+1}=J_{n} =: J_{\infty}$ for every $n\geq n_0$, which is to say that $J_n=w_nJ_nw_{n}^{-1}$ for every $n\geq n_0$. In particular, $w_{n+1} = w_n =: w_{\infty}$ by the lemma above.
 Thus the sequence must stabilize at some point,  $(J_{n+1},w_{n+1})=(J_n,w_n)$ for every $n\geq n_0$. In fact, $n_0$ can be chosen as the smallest index such that $(J_{n+1},w_{n+1})=(J_n,w_n)$. Now let $\Adm_W(J)$ be the set of admissible sequences starting with $J_0=J$ and consider the function
 \begin{align*}
 \Gamma_J \col \Adm_W(J) & \to  W\\
              (J_n,w_n)_{n\geq0}& \mapsto w_{\infty}.
  \end{align*}

 \begin{proposition}[{\cite[Proposition 2.5]{LusztigParabolicI}}]
 \label{prop:Gamma_J_injective}
 The function $\Gamma_J$ is injective and its image is precisely ${}^{J}W$.
  \end{proposition}


 \begin{proposition}
 \label{prop:w_cap}
   Let $J\subset S$ and $w\in {}^JW^J$, then
   \begin{enumerate}
       \item $\Phi_J\cap w\Phi_J=\Phi_{J\cap wJw^{-1}}$.
       \item $(\Phi_J\cap w(\Phi^+\setminus \Phi_J))\cup \Phi^+\setminus\Phi_J=\Phi^+\setminus\Phi_{J\cap wJw^{-1}}$.
       \item $(P_J\cap \dw P_J \dw^{-1})U^J=P_{J\cap wJw^{-1}}$.
   \end{enumerate}
 \end{proposition}
 \begin{proof}
   Since $w\in {}^JW^J$, for each $\beta\in \Phi_J\cap \Phi^+$ we have $w(\beta)\in \Phi^+$ and $w^{-1}(\beta)\in \Phi^+$. Hence, if $\beta\in \Phi_J\cap w(\Phi_J)\cap \Phi^+$, then $\beta=w(\beta')$ with $\beta'\in \Phi_J\cap \Phi^+$. Let $\alpha\in \Delta$ be a simple root satisfying $\alpha <\beta$. Since $\beta\in \Phi_J$, we have $\alpha\in \Delta_J\subset \Phi_J\cap \Phi^+$, hence $w(\alpha)\in \Phi^+$. Moreover, since $\beta = w(\beta')$, $\alpha<w(\beta')$ and both are in $\Phi_J$. Hence, $w(\beta')=\alpha+\gamma$ with $\gamma\in \Phi_J\cap \Phi^+$. This means that $w^{-1}(\alpha)+w^{-1}(\gamma) = \beta'$, where $\beta',w^{-1}(\gamma)\in \Phi_J\cap \Phi^+$ in particular $w^{-1}(\alpha)\in \Phi_J\cap \Phi^+$. We claim that $w^{-1}(\alpha)\in \Delta_J$. Indeed, if there exists $\alpha'\in \Phi_J\cap \Phi^+$ such that $\alpha'<w^{-1}(\alpha)$, then applying $w$, we would have $w(\alpha')\in \Phi^+ $ and $w(\alpha')<\alpha$, which contradicts the fact that $\alpha\in \Delta_J$.\par
    This proves that if $\beta\in \Phi_J\cap w\Phi_J\cap \Phi^+$, then $\beta$ is a sum of simple roots in $\Delta_{J\cap wJw^{-1}}$. Arguing analogously for $\beta\in \Phi_J\cap w\Phi_J\cap \Phi^-$ we see that $\Phi_J\cap w\Phi_J\subset \Phi_{J\cap wJw^{-1}}$. The reverse inclusion is straight-forward. This proves item (1).\par

   To prove item (2), we need only show $\Phi_J\cap w(\Phi^+)=\Phi_J\cap \Phi^+$. Let $\beta\in \Phi_J\cap \Phi^+$. Since $w\in {}^JW^J$, we have $w^{-1}(\beta)\in \Phi^+$ and hence $\beta\in w(\Phi^+)$. This proves that $\Phi_J\cap \Phi^+\subset \Phi_J\cap w(\Phi^+)$. Vice versa, assume that $\beta\in \Phi_J\cap \Phi_-$. Then $w^{-1}(\beta)\in \Phi_-$, which implies that $\Phi_J\cap \Phi_-\subset \Phi_J\cap w(\Phi_-)$. This finishes the proof of item (2).\par

   To prove item (3), write $P_J=L_J U^J$. By \cite[Proposition 2.1]{DigneMichelReps} we have that
   \begin{align*}
       (P_J\cap \dw P_J\dw^{-1})U^J = (L_J\cap \dw L_J\dw^{-1})(L_J\cap \dw U^J\dw^{-1})U^J
   \end{align*}
   By \cite{BT} (see also \cite[Proposition 0.34]{DigneMichelReps}) we have that
   \[
   L_J\cap \dw L_J\dw^{-1}=L_{J\cap w Jw^{-1}}
   \]
   and
   \[
   L_J\cap \dw U^J\dw^{-1}=<U_\beta>_{\beta\in \Phi_J\cap w(\Phi^+\setminus \Phi_J)}.
   \]
   By item (2), we have
   \[
   (L_J\cap \dw U^J\dw^{-1})U^J=<U_\beta>_{\beta\in \Phi^+\setminus\Phi_{J\cap wJw^{-1}}}=U^{J\cap wJw^{-1}}.
   \]
   Hence $(P_J\cap \dw P_J\dw^{-1})U^J=P_{J\cap wJw^{-1}}$.
  \end{proof}

  \begin{proposition}
    Let $g=p_1\dw p_2$ for $p_1,p_2\in P_J$ and $w\in {}^JW^J$ then
    \[
    (P_J\cap gP_Jg^{-1})U^{J}= p_1P_{J\cap wJw^{-1}}p_1^{-1}.
    \]
  \end{proposition}
  \begin{proof}
   Since $p_1^{-1}U^{J}=U^Jp_1^{-1}$, because $U^J$ is normal in $P_J$ (see \cite{AN_hecke}), we have that
   \begin{align*}
       (P_J\cap gP_Jg^{-1})U^J = & (p_1P_Jp_1^{-1}\cap p_1\dw P_J\dw^{-1}p_1^{-1})U^J\\
                               = & p_1(P_J\cap \dw P_j\dw^{-1})p_1^{-1}U^J\\
                               = & p_1(P_J\cap \dw P_j\dw^{-1})U^Jp_1^{-1}.
   \end{align*}
   The result follows from Proposition \ref{prop:w_cap} item (3).
  \end{proof}

\begin{Cons}
\label{cons:lusztig_parabolic}
  Let $J\subset S$ be a subset of simple transpositions of $W$. Consider a pair $(X, gP_J)\in G\times G/P_J$. By the Bruhat decomposition there exists a unique $w\in {}^JW^J$ such that $g^{-1}Xg\in P_J\dw P_J$. Write $g^{-1}Xg = p_1\dw p_2$ and define $J' = J \cap wJw^{-1}$. Define $g' = gp_1$ and consider the pair $(X, g'P_{J'})$. Let $w'$ be the unique element in ${}^{J'}W^{J'}$ such that $g'^{-1}Xg'\in P_{J'}\dw'P_{J'}$. Note that the construction above does not depend on the choice of $g$ in $gP_J$, and we can choose $g$ in such a way that $g^{-1}Xg \in \dw P_J$. In this case, $g' = g$.

\end{Cons}
The analogous construction using flags is given below.
\begin{Cons}
\label{cons:lusztig_parabolic_flag}
Let $G=GL_n$ and let $J\subset S$ be a subset of the set of simple transpositions of $W=S_n$. We identify $S=\{1,\ldots, n-1\}$ and write $S\setminus J = \{ i_1,\ldots, i_k\}$. Let $(V_\bullet^J,X)$ be a pair in $G\times G/P_J$, where
\[
V_\bullet^J = 0 = V_0 \subset V_{i_1}\subset V_{i_2}\subset \ldots \subset V_{i_k}\subset V_n=\mathbb{C}^n.
\]
We construct the flag $V'_\bullet$ by coalescing (we remove vector spaces with the same dimension) the flag
\begin{align*}
   &0 \subseteq V_{i_1}\cap XV_{i_1} \subseteq V_{i_1}\cap XV_{i_2} \subseteq \ldots V_{i_1}\cap XV_{i_k}\subseteq V_{i_1}\subseteq \\
   &\subseteq V_{i_1}+(V_{i_2}\cap X{V_{i_1}}) \subseteq V_{i_1}+(V_{i_2}\cap XV_{i_2}) \subseteq\ldots V_{i_1}+(V_{i_2}\cap XV_{i_k})\subseteq V_{i_2}\subseteq \\
   & \vdots \\
   &\subseteq V_{i_j}+(V_{i_{j+1}}\cap XV_{i_1})\subseteq V_{i_j}+(V_{i_{j+1}}\cap XV_{i_2})\subseteq \ldots \subseteq V_{i_j}+(V_{i_{j+1}}\cap XV_{i_{k}}) \subseteq V_{i_{j+1}} \subseteq\\
   &\vdots \\
   &\subseteq V_{i_k}+ XV_{i_1}\subseteq V_{i_k}+ XV_{i_2}\subseteq \ldots \subseteq V_{i_k}+ XV_{i_{k}} \subseteq V_{n} =\mathbb{C}^n.
  \end{align*}

As in Construction \ref{cons:lusztig_parabolic}, there exists a unique permutation $w\in {}^JS_n^J$ such that $\dim XV_{i_j}\cap V_{i_{j'}}=r_{i_j,i_{j'}}(w)$ . Moreover, there exists $J'$ such that $V'_\bullet$ is a partial flag in $G/P_{J'}$ and we can construct $w'$ analogously.

\end{Cons}

\begin{proposition}
 For $G=GL_n$, constructions \ref{cons:lusztig_parabolic} and \ref{cons:lusztig_parabolic_flag} are equivalent.
\end{proposition}

  \begin{proof}

    $G/P_{J}$ can be identified with the partial flag variety parametrizing flags
    \[
    V^J_\bullet = 0=V_0 \subset V_{i_1}\subset \ldots \subset V_{i_l} \subset \ldots V_{n}=\mathbb{C}^n
    \]
    where $\{i_1,\ldots, i_k\}=\{1,\ldots,n-1\}\setminus J=: J^c$. We briefly recall this identification. If a matrix $g$ has $v_1,\ldots, v_n$ as columns, then the point $gP_J$ is associated to the partial flag $V_\bullet$ where $V_{i_k}=\langle v_1,\ldots, v_{i_k}\rangle$. If $X$ is an invertible matrix, then $XgP_J$ corresponds to the flag $XV_\bullet^J = (XV_{i_k})$. The condition that $g^{-1}Xg\in P_J\dw P_J$ for $w\in {}^JS_n^J$ is equivalent to the equalities $\dim V_{i_k}\cap XV_{i_{k'}} = r_{i_k,i_{k'}}(w)$ (in fact, there exists exactly one permutation $w\in {}^JS_n^J$ satisfying these equalities).\par
    It is clear that given a flag $V^J_\bullet$ there exist multiple choices of $g$ such that $gP_J$ induces $V^J_\bullet$. In fact any element in $gP_J$ will give rise to the same flag. As in Construction \ref{cons:lusztig_parabolic}, we can choose $g$ in a way that $g^{-1}Xg\in \dw P_J$. This is equivalent to saying that $Xg$ and $g \dw$ induces the same partial flag $XV^J_\bullet$. In other words, we have that $\langle Xv_1,\ldots, Xv_{i_j} \rangle = \langle v_{w(1)},v_{w(2)},\ldots, v_{w(i_j)} \rangle$ for every $j\in 1,\ldots, k$.\par

    Before continuing, let us characterize the permutations in ${}^JS_n^J$. We have that $w\in {}^JS_n^J$ if and only if $w(i)<w(i+1)$ and $w^{-1}(i)<w^{-1}(i+1)$ for every $i \in J$. Also, we have that $J_1=J\cap wJw^{-1}$ consists precisely of the elements $j$ of $J$ such that there exists $i\in J$ with $j=w(i)$ and $j+1=w(i+1)$.\par

    Let $J'$ and $g'$ be as in Construction \ref{cons:lusztig_parabolic}. We will see that $g'P_{J'}$ induces the flag $V'_\bullet$ given by Construction \ref{cons:lusztig_parabolic_flag}. In fact, we have that $g'=g$ (recall that we chose $g$ satisfying $g^{-1}Xg\in \dw P_J$). We let $F_\bullet$ be the partial flag induced by $g'P_{J'}$. Writing $\{i'_1,\ldots, i'_{k'}\}:=\{1,\ldots, n-1\}\setminus J'$, we have that $F_{i'_j}=\langle v_1,\ldots, v_{i'_j} \rangle$. If $i'_j\notin J$ then $F_{i'_j}=V_{i'_j} = V'_{i'_j}$. Let us assume $i'_j\in J$, in particular that $w^{-1}(i'_j)< w^{-1}(i'_j +1)$. Let $i_{j_1}$ be the maximum element of $\{0\} \cup (J^c\cap \{1,\ldots, i'_k\})$ and let $i_{j_2}$ be an element of $\{w^{-1}(i'_j),w^{-1}(i'_{j})+1,\ldots, w^{-1}(i'_j+1)-1\}\setminus J$. We will see that
    \begin{equation}
    \label{eq:FVX}
      F_{i'_j} = V_{i_{j_1}}+(V_{i_{j_1+1}}\cap XV_{i_{j_2}}).
    \end{equation}

    First, let us prove that $\{w^{-1}(i'_j),w^{-1}(i'_{j})+1,\ldots, w^{-1}(i'_j+1)-1\}\setminus J$ is nonempty.  Assume by contradiction that $\{w^{-1}(i'_j),w^{-1}(i'_{j})+1,\ldots, w^{-1}(i'_j+1)-1\}\subset J$. Since $w(i)<w(i+1)$  for every $i\in J$, we have that
    \[
      i'_j<w(w^{-1}(i'_j)+1)<w(w^{-1}(i'_j)+2)<\ldots <w(w^{-1}(i'_j+1)-1)< i'_j+1,
    \]
    which can only happen if $w^{-1}(i'_j+1)= w^{-1}(i'_j)+1$, but that implies that $i'_j\in wJw^{-1}$ and hence $i'_j\in J_1$, a contradiction.


    Finally, let us prove Equation \eqref{eq:FVX}. We have that $F_{i'_j}=\langle v_1,\ldots v_{i'_j}\rangle$, $V_{i_{j_1}}=\langle v_1,\ldots, v_{i_{j_1}} \rangle$, $V_{i_{j_1+1}}=\langle v_1,\ldots, v_{i_{j_1+1}} \rangle $ and $XV_{i_{j_2}}=\langle v_{w(1)}, v_{w(2)},\ldots, v_{w(i_{j_2})}\rangle$. Then it is enough to prove that
    \begin{equation}
    \label{eq:iw}
    \{1,\ldots, i'_j\} = \{1,\ldots, i_{j_1}\} \cup (\{1,\ldots, i_{j_1+1}\} \cap \{w(1),\ldots, w(i_{j_2})\} ).
    \end{equation}
    By the definition of $i_{j_1}$, we have that $i_{j_1}<i_{j_1}+1 <\ldots < i'_{j} < i'_j+1<\ldots < i_{j_1+1}$ hence $\{i_{j_1}+1,\ldots, i_{j_1+1}-1\}\subset J$ and then
    \[
    w^{-1}(i_{j_1}+1)<w^{-1}(i_{j_1}+2)<\ldots < w^{-1}(i'_j) < w^{-1}(i'_j+1) < \ldots <w^{-1}(i_{j_1+1}).
    \]
    Since $w^{-1}(i'_j)\leq i_{j_2} < w^{-1}(i'_j+1)$, we have that
    \begin{equation}
        \label{eq:i'k}
    w^{-1}(i_{j_1}+1)<w^{-1}(i_{j_1}+2)<\ldots < w^{-1}(i'_j)\leq i_{j_2} < w^{-1}(i'_j+1) < \ldots <w^{-1}(i_{j_1+1}).
    \end{equation}
    This means that $\{i_{j_1}+1,\ldots, i'_j\}\subset \{w(1),\ldots, w(i_{j_2})\}$ while $\{i'_j+1,\ldots, i_{j_1+1}\}\cap \{w(1),\ldots, w(i_{k_2})\}=\emptyset$, proving Equation \eqref{eq:iw} (and, hence, Equation \eqref{eq:FVX} as well).\par

    Note that any choice of an element $i_{j_2}\in \{w^{-1}(i'_j),w^{-1}(i'_j)+1,\ldots, w^{-1}(i'_j+1)-1\}\setminus J$ will satisfy Equation \eqref{eq:FVX}. Conversely, given any $i_{j_1}, i_{j_2} \in J^c$, the set
    \[
    \{1,\ldots, i_{j_1}\} \cup (\{1,\ldots, i_{j_1+1}\} \cap \{w(1),\ldots, w(i_{j_2})\}
    \]
    is of the form $\{1,\ldots, i'_j\}$. Indeed, just choose $i'_j$ such that Equation \eqref{eq:i'k} holds, if it exists. If it does not exits, then it means that $i_{j_2}<w^{-1}(i_{j_1}+1)$. If $i'_j$ does not exists or $i'_j=i_{j_1}+1$, then
    \[
     V_{i_{j_1}}+(V_{i_{j_1+1}}\cap XV_{i_{j_2}}) = F_{i_{j_1}}\text{ or } F_{i_{j_1+1}}.
    \]
    Otherwise, we have that $w^{-1}(i'_k)\leq i_{k_2} < w^{-1}(i'_k+1)$, which means that $i'_k\notin J_1$ and hence
    \[
    V^J_{i_{k_1}}\cup(V^{J}_{i_{k_1+1}}\cap XV_{i_{k_2}}) = V^{J_1}_{i'_k}.
    \]

  \end{proof}

Given a pair $(X, g_0P_{J_0})$ we have an associated pair $(J_0,w_0)$ with $w_0\in {}^{J_0}W^{J_0}$ such that $g^{-1}Xg\in P_{J_0}\dw P_{J_0}$. We can iterate Construction \ref{cons:lusztig_parabolic} and construct a sequence $(X,g_nP_{J_n})$ with an associated sequence $(J_n, w_n)$. We claim that the latter sequence is admissible.


  \begin{proposition}[{\cite[Section 2.8]{LusztigParabolicI}}]
    Starting with a pair $(X, g_0P_{J_0})$ and constructing the sequence $(J_n,w_n)$ following the steps above, we have that the sequence $(J_n,w_n)$ is admissible.
  \end{proposition}



  \begin{definition}
  \label{def:lusztig_parabolic}
    For every sequence $\mathbf{t}=(J_n,w_n)$ with $J_0=J$, we define the \emph{parabolic Lusztig cell} and \emph{variety} as
    \begin{align*}
    \mathcal{Y}_{\mathbf{t}}^\circ&:=\{(X, gP_J), \mathbf{t}(X, gP_J)=\mathbf{t}\},\\
    \mathcal{Y}_{\mathbf{t}} &:= \overline{\mathcal{Y}_{\mathbf{t}}^\circ},
    \end{align*}
    respectively.    Since every sequence $\mathbf{t}$ with $J_0=J$ is in bijection with ${}^JW$, we write $\mathcal{Y}_{w,J}^{\circ}:=\mathcal{Y}_{\mathbf{t}}^\circ$ and $\mathcal{Y}_{w,J} := \mathcal{Y}_{\mathbf{t}}$.
  \end{definition}

\begin{example}
  If $W=S_n$ and $w\in {}^JW$ is such that $w^{-1}Jw=J$ (in particular $w\in {}^JW^J$), then there is only one admissible sequence starting with $(w,J)$, and it is the constant sequence $(w_n,J_n)=(w,J)$. This means that we have the following characterization of $\h_{w,J}^\circ$
\begin{align*}
    \h_{w,J}^\circ:=\{(X, V_\bullet); X\in G, XV_i\cap V_j = r_{ij}(w)\text{ for }i,j\in [n]\setminus J\}.
\end{align*}

\end{example}

  \begin{proposition}[{\cite[Section 3.11]{LusztigParabolicI}}]
  \label{prop:loc_closed}
   The variety $\Y_{\mathbf{t}}^{\circ}$ is a locally closed subvariety of $G/P_J\times G$.
  \end{proposition}

  \begin{proposition}[{\cite[Section 3.12]{LusztigParabolicI}}]
  \label{prop:t_iso_t1}
      We have an isomorphism $\h_{\mathbf{t}}^\circ\to \h_{\mathbf{t}_1}^\circ$ where $\mathbf{t}_1 = (J_n,w_n)_{n\geq 1}$. In particular, $\h_{w,J}^\circ$ is naturally isomorphic to $\h_{w, J_w}^\circ$.
  \end{proposition}

  \begin{remark}
  In view of Construction \ref{cons:lusztig_parabolic_flag} the isomorphism $\h_{\mathbf{t}}^\circ\to \h_{\mathbf{t}_1}^\circ$ is given by $(V_\bullet^{J_0}, X)\mapsto (V'_\bullet, X)$. Essentially, a pair $(X, V_\bullet^{J_0})$ is in $\h_{\mathbf{t}}^\circ$ if, when we apply Construction \ref{cons:lusztig_parabolic_flag}, the flag $V'_\bullet$ always lies in the same partial flag variety and has the same intersection numbers $\dim XV'_i\cap V'_j$. Moreover, this remains true after iterating the construction for $(X,V'_\bullet)$.\par
  \end{remark}

  \begin{Rem}
  \label{rem:w_double}
  If $J$ and $w$ satisfy $w\in {}^JW^J$ and $wJw^{-1}=J$, then the sequence $\mathbf{t}=(J_n,w_n)_{n\geq 0}$ given by $J_n=J$ and $w_n=w$ for every $n\geq 0$ is admissible. In this case, we have that
  \[
  \h_{w,J}^\circ =\h_\mathbf{t}^\circ=\{(X,gP_J); g^{-1}Xg\in P_J\dw P_J\}.
  \]
  Moreover, applying Proposition \ref{prop:t_iso_t1} repeatedly, we have that $\h_{w,J}^\circ$ is isomorphic to $\h_{w,J_w}^{\circ}$. This means that every pair $(X,V_\bullet)\in \h_{w,J}^\circ$ is obtained by the restriction of a pair $(X,\overline{V}_\bullet)\in \h_{w,J_w}^\circ$.
  \end{Rem}

  \begin{Exa}
     Let $G=GL_4(\mathbb{C})$ and $J=\{1,3\}$, so $G/P_J$ is the Grassmanian $Gr(2,4)$ and fix a regular semisimple matrix $X\in G$. The admissible sequences are
     \[
     \begin{array}{lll}
      \big( (\{1,3\}, 1234) \big);  &  \big( (\{1, 3\}, 1324), (\emptyset, 1324) \big); & \big(  (\{1,3\}, 1324), (\emptyset, 1342)\big);\\
      \big( (\{1,3\}, 1324), (\emptyset, 3124)\big); & \big( (\{1,3\}, 1324), (\emptyset, 3142)\big); &   \big( (\{1, 3\},3412)\big).
     \end{array}
     \]
     Note that the permutations $1234, 1324, 1342, 3124, 3142, 3412$ are precisely the elements of ${}^JS_4$. Let us describe the varieties $\h_{w,J}$ for $w\in {}^JS_4$. We begin with $w=1234, 3412$ (which are the permutations in ${}^JS_4^J$. Since the sequences have only one term, then $\h_{w,J}^\circ(X)=\{V_2; XV_2\cap V_2= r_{2,2}(w)\}$, hence
     \[
     \h_{1234,J}^{\circ}(X)=\{V_2; XV_2=V_2\}\quad \text{ and }\quad \h_{3412,J}^\circ(X) = \{ V_2; \dim XV_2\cap V_2 =0\}.
     \]
     All other sequences start with $(\{1,3\}, 1324)$ so we must have the condition $\dim XV_2\cap V_2=1$. Next, we use construction \ref{cons:lusztig_parabolic_flag} to construct the flag $V'_\bullet$ associated to the pair $(V_2,X)$. We have that
     \[
     V'_\bullet = 0\subset V'_1=V_2\cap XV_2 \subset V'_2=V_2 \subset V'_3=V_2+XV_2\subset \mathbb{C}^4.
     \]
     So we have that
     \[
     \h_{w,J}^\circ(X) =\{V_2; V'_\bullet \in \h_{w,\emptyset}^\circ\}= \{V_2; \dim XV'_i\cap V'_j=r_{i,j}(w)\}
     \]
     for $w\in \{1324, 1342, 3124, 3142\}$. Therefore, we have
     \begin{align*}
     \h_{1324,J}^\circ(X)=&\{ V_2;  \dim XV'_1\cap V'_1=1,\dim XV'_2\cap V'_2 =1, \dim XV'_3\cap V'_3 =3\}\\
                =&\bigg\{ V_2; \begin{array}{l}\dim X^2V_2\cap XV_2\cap V_2=1, \dim XV_2\cap V_2=1,\\ \dim (XV_2+X^2V_2)\cap (V_2+XV_2)=3\end{array}\bigg\}\\
                =&\{ V_2; \dim X^2V_2\cap XV_2\cap V_2=1,  \dim X^2V_2+XV_2+V_2=3\},
     \end{align*}
     \begin{align*}
     \h_{1342,J}^\circ(X)=&\{ V_2;  \dim XV'_1\cap V'_1=1, \dim XV'_2\cap V'_3=2, \dim XV'_3\cap V'_2=1 \}\\
                =&\{ V_2; \dim X^2V_2\cap XV_2\cap V_2=1, \dim (XV_2+X^2V_2)\cap V_2=1\}\\
                =&\{ V_2; \dim X^2V_2\cap XV_2\cap V_2=1,  \dim X^2V_2+XV_2+V_2=4\}.
     \end{align*}
     The last two are
     \begin{align*}
         \h_{3124,J}^\circ(X)=&\{ V_2; \dim X^2V_2\cap XV_2\cap V_2=0, \dim X^2V_2+ XV_2 + V_2=3\}\\
         \h_{3142,J}^\circ(X)=&\{ V_2;\dim XV_2\cap V_2=1, \dim X^2V_2\cap XV_2\cap V_2=0, \dim X^2V_2+ XV_2 + V_2=4\}.
     \end{align*}
    We note that, for $w\in \{1324, 1342, 3124, 3142\}$, the varieties $\h_{w,J}^\circ(X)$ are precisely the image of $\h_{w}^\circ(X)$  via the forgetful map $\flag_4\to \Gr(2,4)$ (see Remark \ref{rem:w_double}).

     Let us compare these varieties with the Schubert varieties of $\Gr(2,4)$. Fix a flag
\[
F_\bullet= (F_1\subset F_2\subset F_3\subset F_4=\mathbb{C}^4).
\]
Instead of using permutations, we can use partitions in the rectangle $2\times 2$ to index both the Schubert and Lusztig varieties in $\Gr(2,4)$. Below, we write both varieties for each partition.
\begin{align*}
    \h_{\emptyset, J}(X)&=\{V_2; XV_2=V_2\}, & \Omega_{\emptyset,J,F_\bullet}&=\{V_2; V_2=F_2\},\\
    \h_{\yd{1}, J}(X)&=\bigg\{V_2; \begin{array}{l}\dim(V_2\cap XV_2\cap X^2V_2)\geq 1\\\dim(V_2+ XV_2+ X^2V_2)\leq 3\end{array}\bigg\}, &  \Omega_{\yd{1,0} ,J,F_\bullet}&=\{V_2; F_1\subset V_2\subset F_3\}, \\
    \h_{\yd{2}, J}(X)&=\{V_2;  \dim(V_2\cap XV_2\cap X^2V_2)\geq 1\},& \Omega_{\yd{2,0},J,F_\bullet}&=\{V_2; F_1\subset V_2\},\\
    \h_{\yd{1,1}, J}(X)&=\{V_2; \dim(V_2+ XV_2+ X^2V_2)\leq 3  \}, & \Omega_{\yd{1,1},J,F_\bullet}&=\{V_2;  V_2\subset F_3\},\\
    \h_{\yd{2,1}, J}(X)&=\{V_2; \dim(XV_2\cap V_2)\geq 1  \}, & \Omega_{\yd{2,1},J,F_\bullet}&=\{V_2; \dim( V_2\cap F_2)\geq 1\},\\
    \h_{\yd{2,2}, J}(X)&=\Gr(2,4), & \Omega_{\yd{2,2}, J,F_\bullet}&=\Gr(2,4).\\
\end{align*}
The varieties $\h_{\emptyset,J}(X), \h_{\yd{1}}(X)$, $\h_{\yd{2}}(X)$ and $\h_{\yd{1,1}}(X)$ are actually finite unions of translations of Schubert varieties. Indeed,  the condition $\dim(V_2\cap XV_2\cap X^2V_2)\geq 1$ is equivalent to the fact that $V_2$ contains a 1-dimensional subspace invariant by $X$ (and there is only a finite number of them, because $X$ is regular). Analogously, $\dim(V_2+XV_2+X^2V_2)\leq 3$ is equivalent to the fact that $V_2$ is contained in a $3$-dimensional subspace invariant by $X$ and the condition $XV_2=V_2$ is equivalent to the fact that $V_2$ is a subspace invariant by $X$.  Below we give a more detailed description of the geometry of parabolic Lusztig varieties in the Grassmanian $\Gr(2,4)$.
     \begin{enumerate}
         \item The variety $\h_{\emptyset,J}(X)$ is the collection of the $6$ points fixed by the torus action on $\Gr(2,4)$. The Schubert variety $\Omega_{\emptyset, F_\bullet}$ is the single point $V_2 = F_2$.
         \item The variety $\h_{\yd{1}, J}(X)$ is the union of $12$ copies of $\mathbb{P}^1$, where each one is of the form $\{V_2; \langle e_i\rangle\subset V_2\subset \langle e_i,e_j,e_k\rangle\}$ where $e_1,e_2, e_3, e_4$ are the eigenvectors of $X$. The Schubert Variety $\Omega_{\yd{1},F_\bullet}$ is isomorphic to a single $\mathbb{P}^1$.
         \item The variety $\h_{\yd{2},J}(X)$ ($\h_{\yd{1,1}, J}(X)$, respectively) is the union of $4$ copies of $\mathbb{P}^2$, each of the form $\{V_2, \langle e_i\rangle\subset V_2\}$ ($\{V_2; V_2\subset \langle e_i, e_j, e_k\rangle\}$, respectively). The Schubert Varieties $\Omega_{\yd{2}, F_\bullet}$ and $\Omega_{\yd{1,1}, F_\bullet}$ are isomorphic to $\mathbb{P}^2$.
         \item The variety $\h_{\yd{2,1},J}(X)$ can be described as follows. Consider the forgetful  map from the variety \[
         \h_{2341, \{3\}}(X) = \{V_1\subset V_2; XV_1\subset V_2\}
         \]
         to the Grassmanian $\Gr(2,4)$. The image is precisely $\h_{\yd{2,1},J}(X)$ and the map is birational, because for a generic $V_2\in \h_{\yd{2,1},J}(X)$ there exists only one $V_1$ satisfying $XV_1\subset V_2$, which is $V_1=V_2\cap X^{-1}V_2$. Moreover, the preimages of the points $V_2\in \h_{\emptyset,J}(X)$ (which are the points that satisfy $V_2=XV_2$) are isomorphic to $\mathbb{P}^1$. \par
         As seen in Example \ref{exa:path}, we have that $\h_{2341, {3}}(X)$ is obtained by blowing up $\mathbb{P}^3$ at the $4$ special points. So, $\h_{\yd{2,1},J}(X)$ can be obtained by blowing up $\mathbb{P}^3$ at the $4$ special points and contracting the strict transforms of the $6$ special lines. In particular the map $BL_{p_1,p_2, p_3, p_4}\mathbb{P}^3\to \h_{\yd{2,1},J}(X)$ is small. \par
          The Schubert variety $\Omega_{\yd{2,1}, F_\bullet}$ is a cone over the quadric surface. It can be constructed as the contraction of of the rational curve $\{V_1\subset V_2; V_2 = F_2\}$ in the Schubert variety $\Omega_{3142,\{3\}, F_\bullet} = \{V_1\subset V_2; V_1\subset F_2\}$. The map $\Omega_{3142, \{3\}, F_\bullet} \to \Omega_{\yd{2,1},F_\bullet}$ is small and the domain is a $\mathbb{P}^2$-bundle over $\mathbb{P}^1$.

     \end{enumerate}

Schubert varieties have the following property: if $w$ is a permutation in $S_n$ and $w_0\in {}^JS_n$ is such that $w\in w_0S_n$, then the image of $\Omega_{w,\emptyset, F_\bullet}^\circ$ is precisely $\Omega_{w_0, J, F_\bullet}^\circ$. In particular if $f\col \flag_4\to Gr(2,4)$ is the forgetful map $f(V_\bullet)=V_2$, then we have that for each $w\in S_4$ there exists a partition $\lambda_w$ inside the $2\times 2$ square, such that $f(\Omega_{w,F_\bullet})=\Omega_{\lambda,F_{\bullet}}$. We abuse notation and denote by $f\col S_4\to \parti(4)$ the function defined by $f(w)=\lambda_w$. We have that
  \begin{align*}
      f^{-1}(\emptyset)&=\{1234, 2134, 1243, 2143\},\\
      f^{-1}(\yd{1})&=\{1324,2314,1423,2413\},\\
      f^{-1}(\yd{2,0})&=\{1342,2341,1432,2431\},\\
      f^{-1}(\yd{1,1})&=\{3124,3214,4132,4213\},\\
      f^{-1}(\yd{2,1})&=\{3142,3241,4132,4231\},\\
      f^{-1}(\yd{2,2})&=\{3412,3421,4312,4321\}.
  \end{align*}
These are precisely the right cosets $S_2\times S_2\backslash S_4$ .\par

      However, for Lusztig varieties the situation is more complicated. First, it can happen that the image of $\h_{w}^\circ(X)$ via the forgetful map is not any of $\h_{w, J}^\circ(X)$. Second, even though the image of $\h_w(X)$ is $\h_{w_0,J}(X)$ for some $w_0\in {}^JS_n$, it is no longer true that $w_0$ is the minimum representative of $w$. For each $w\in S_4$, there exists a partition $\mu_w$ such that $f(\h_w(X))=\h_{\mu_w}(X)$. It can happen that $\mu_w\neq \lambda_w$. For instance, if $w=2314$, then
  \[
  \h_{2314}(X)=\{V_\bullet; XV_1\subset V_2, XV_3=V_3\}
  \]
  and
  \[
  f(\h_{2314}(X))=\{V_2; \dim(V_2+ XV_2+ X^2V_2)\leq 3  \}=\h_{\yd{1,1}}(X).
  \]

  Indeed, since $XV_3=V_3$, $V_2$ is contained in a $3$-dimensional subspace invariant by $X$. Moreover, for each $V_2$ inside such a $3$-dimensional subspace, we can always choose $V_1\subset V_2\cap X^{-1}V_2$, because $\dim (V_2\cap X^{-1}V_2)=4-\dim(V_2+XV_2)\geq 1$. If $g\col S_4\to \parti(4)$ is the function given by $g(w)=\mu_w$, then
   \begin{align*}
      g^{-1}(\emptyset)&=\{1234, 1243, 2134, 2143\},\\
      g^{-1}(\yd{1})&=\{1324\},\\
      g^{-1}(\yd{2,0})&=\{1342, 1423, 1432\},\\
      g^{-1}(\yd{1,1})&=\{2314, 3124, 3214\},\\
      g^{-1}(\yd{2,1})&=\{2341, 2413, 2431, 3142, 3241, 4123, 4132, 4213, 4231  \},\\
      g^{-1}(\yd{2,2})&=\{3412, 3421, 4312, 4321\}.
  \end{align*}

    Using the geometric description above, we see that $IH^*(\h_{\lambda, J}(X))$ has a natural $S_4$-module structure and its Frobenius characters can be computed. The result is:
  \begin{align*}
    \ch(IH^*(\h_{\emptyset}(X)))&=h_{2,2},\\
    \ch(IH^*(\h_{\yd{1}}(X)))&=(1+q)h_{2,1,1},\\
    \ch(IH^*(\h_{\yd{2}}(X)))&=(1+q+q^2)h_{3,1},\\
    \ch(IH^*(\h_{\yd{1,1}}(X)))&=(1+q+q^2)h_{3,1},\\
    \ch(IH^*(\h_{\yd{2,1}}(X)))&=(1+q+q^2+q^3)h_4+(q+q^2)h_{3,1}, \\
    \ch(IH^*(\h_{\yd{2,2}}(X)))&=(1+q+2q^2+q^3+q^4)h_4=\binom{4}{2}_qh_4.
  \end{align*}
  As for the case of complete flag varieties (see \cite[Equation 2.1]{Haiman}),
  the sum of the $h$-coefficients is the Poincaré polynomial of the associated Schubert variety. Indeed, we have that $\Omega_{\yd{1},F_\bullet}$ is isomorphic to $\mathbb{P}^1$, $\Omega_{\yd{2},F_\bullet}$ and $\Omega_{\yd{1,1},F_\bullet}$ are isomorphic to $\mathbb{P}^2$, $\Omega_{\yd{2,1},F_\bullet}$ is a cone over $\mathbb{P}^1\times \mathbb{P}^1$ (in particular, its small resolution is a $\mathbb{P}^2$-bundle over $\mathbb{P}^1$ and hence its Poincaré polinomial is $(1+q)(1+q+q^2)=1+2q+2q^2+q^3$), and $\Omega_{\yd{2,2},F_\bullet}$ is the Grassmanian $Gr(2,4)$ that has Poincaré polinomial equal to $\binom{4}{2}_q$.\par

  \end{Exa}

\begin{Exa}[The projective space]
\label{exa:projective}
When $G=GL_n$ and $J=\{2,\ldots, n-1\}$, we have that $G/P_J=\mathbb{P}^{n-1}$ and
\[
{}^JS_n = \{12\ldots n,213\ldots n, 231\ldots n,\ldots ,234\ldots n1\}.
\]
Define $w_k:=23\ldots k 1\ldots n$ ($w_1$ is the identitiy). We also have that $J_{w_k}=\{k+1,\ldots, n-1\}$. In particular, we have
\[
\h_{w,J_{w_k}} = \{(X,V_1\subset V_2\subset \ldots\subset V_k); XV_k=V_k, XV_i\subset V_{i+1}\},
\]
and hence
\[
\h_{w,J} = \{(X,V_1); \dim (V_1 +XV_1+\ldots+X^kV_1 )\leq k\},
\]
or, equivalently,
\[
\h_{w,J} = \{(X,V_1); V_1\text{ is contained in a dimension $k$ subspace invariant by $X$}\}.
\]
When $X$ is a diagonal matrix, the condition that $V_1$ is contained in a dimension $k$ subspace invariant by $X$ is equivalent to $V_1\in \mathbb{P}^{n-1}$ belonging to a dimension $k-1$ coordinate plane. In particular, this proves that $\h_{w_k,J}(X) = \Hi_{n-k}$ (see \ref{cor:chHi}).

The normalization $\widetilde{\Hi}_{i}$ of $\Hi_{i}$ is the union of $\binom{n}{i}$ copies of $\mathbb{P}^{n-i-1}$. Hence the intersection cohomology of $\Hi_{n-k}$ is the cohomology  $H^*(\widetilde{\Hi}_{i})$. The latter has a natural structure of $S_n$-module and
\[
\ch(H^*(\widetilde{\Hi}_{i})) = [n-i]_qh_{n-i, i}.
\]
\end{Exa}

\section{Monodromy actions on parabolic varieties}
\label{sec:monodromy}

The goal of this section is to prove the following.

\begin{theorem}
\label{thm:pi1_WJw}
 Let $G=GL_n$, $w\in S_n$  and $J\subset S$ be such that $w\in {}^JW$ and $wJw^{-1}=J$. There exists an open set $\U_{w,J}\subset \h_{w,J}$ and a natural map
 \[
\pi_1(\U_{w,J}, (X,gP_J))\to W_J^w:=\{z\in W_J; wz=zw\}.
\]
In particular, every representation $W_J^w\to GL(V)$ of $W_J^w$ induces a local system $L$ on $\U_{w,J}$ with fiber $V$, and this local system induces a perverse sheaf $IC_{\h_{w,J}}(L)$ on $\h_{w,J}$. Moreover, if $f\colon \h_{w}\to \h_{w,J}$ is the forgetful map and $g=f|_{f^{-1}(\U_{w,J})}$, then $g$ is a $W_J^w$-Galois cover, which means that $g_*(\mathbb{C}_{f^{-1}(U_{w,J})})$ is the local system on $\U_{w,J}$ induced by the representation of $W_J^w$ on itself.
\end{theorem}
This theorem implies the first statement of Theorem \ref{thm:UwJ_intro}.

\begin{example}
\label{exa:G24_not_galois}
Consider $w=3412$ and $J=\{1,3\}$, then $\h_w=\{(X,V_\bullet); XV_1\subset V_3, V_1\subset XV_3\}$ and the map $f$ in the theorem is
\begin{align*}
    f\colon \h_{w}&\to \h_{w,J}= G\times \Gr(2,4)\\
    (X, V_\bullet)&\to (X, V_2).
\end{align*}
We note that even on the locus $G^{rs}\times \Gr(2,4)^\circ$, the restriction of $f$ is not a Galois cover (nor even a cover at all). Indeed, take $X$ to be the diagonal matrix with entries $(1,2,-1,-2)$, define the vectors $v_1=(1,1,1,1)$, $v_2=(1,-1,0,0)$, $v_3=Xv_1=(1,2,-1,-2)$ and $v_4=Xv_2=(1,-2,0,0)$, and let $V_2=\langle v_1,v_2\rangle$. Then $v_1,v_2,v_3,v_4$ are linearly independent, and hence $XV_2\cap V_2 = 0$, so $(X,V_2)\in G^{rs}\times \Gr(2,4)^\circ$.\par
 To find the preimage $f^{-1}(X,V_2)$, we have to find $((a:b),(c:d))\in \mathbb{P}^1\times \mathbb{P}^1$ such that $XV_1\subset V_3$ and $V_1\subset XV_3$ where $V_1=\langle av_1+bv_2\rangle$ and $V_3=\langle v_1,v_2,cv_3+dv_4\rangle$. That is, we must have that the determinants $\det(X(av_1+bv_2), v_1, v_2,cv_3+dv_4)$ and $\det((av_1+bv_2), Xv_1, Xv_2, X(cv_3+dv_4))$ are $0$. Writing the equations
 \[
 \left | \begin{array}{cccc}
 a+b    & 1 & 1  &  c + d   \\
 2a- 2b & 1 & -1 &  2c - 2d \\
 -a     & 1 & 0  &  -c      \\
 -2a     & 1 & 0  &  -2c     \\
 \end{array}
 \right |= 0\quad\quad\text{ and }\quad \quad
 \left | \begin{array}{cccc}
 a+b   & 1  & 1  &  c + d   \\
 a- b  & 2  & -2 &  4c - 4d \\
 a     & -1 & 0  &  c       \\
 a     & -2 & 0  &  4c      \\
 \end{array}
 \right | = 0,
  \]
More explicitly, we get the equations
\begin{align*}
    bc-ad &= 0\\
    -24ac - 2bc + 2ad&=0,
\end{align*}
which correspond to the double point $((0:1),(0:1))\in \mathbb{P}^1\times \mathbb{P}^1$. This means that the map $f$ is ramified over $(X,V_2)$.
\end{example}

Although the theorem above is only for $G=GL_n$, some of the results below hold for any $G$. Throughout this section, we fix $J\subset S$ and $w\in {}^JW^J$ such that $wJw^{-1}=J$. With these conditions, we have that $\dw L_J \dw^{-1} = L_J$, in particular we have that $\dw L_J$ is a component of $N_G(L_J)$. We define $T_1:=(T^w)^{0}=(\{t\in T; \dw t=t\dw\})^0$ (here $G^0$ means the identity component of $G$). We let $\N_{L_J}:=\{\ell \in L_J; \ell wT_1\ell^{-1}=wT_1\}$ and $\N_T:=\{t\in T; twT_1t^{-1}=wT_1\}$. We also define $W_J^w:=\{z\in W_J; wz=zw\}$. \par

 \begin{Exa}
 Let $G=GL_4$, $J=\{1,3\}$ and consider $w = 3412$. Then $wJw^{-1} = J$. The torus $T_1$ is the torus of diagonal matrices whose diagonal is of the form $(t_1, t_2, t_1, t_2)$. In particular, we have
 \[
 wT_1=\Bigg\{ \left( \begin{array}{cccc}
  0 & 0 & t_1 & 0 \\
  0 & 0 & 0 & t_2 \\
  t_1 & 0 & 0 & 0 \\
  0 & t_2 & 0 & 0
 \end{array}\right )\Bigg\}.
 \]
 The group $\N_T$ is disconnected with four components  consisting of diagonal matrices with diagonal $(\lambda_1, \lambda_2, \pm \lambda_1, \pm \lambda_2)$.
 The group $\N_{L_J}$ is also disconnected and has eight components, namely
 \[
 \N_{L_J} = \Bigg\{ \left( \begin{array}{cccc}
  a & 0 & 0 & 0 \\
  0 & d & 0 & 0 \\
  0 & 0 & \pm a & 0 \\
  0 &0 & 0 & \pm d
 \end{array}\right )\Bigg\} \cup \Bigg\{ \left( \begin{array}{cccc}
  0 & b & 0 & 0 \\
  c & 0 & 0 & 0 \\
  0 & 0 & 0 & \pm b  \\
  0 &0 & \pm c & 0
 \end{array}\right )\Bigg\}.
 \]
 The identity components $(\N_{L_J})^0$, $(\N_T)^0$ coincide with $T_1$. Moreover, $\N_{L_J}/ \N_T$ is naturally identified with the subgroup of $S_4$ generated by $2143$, which is the same as $(S_4)_J^{3412}$. Also, we note that the normalizer $N_G(L_J)$ is equal to $L_J \cup \dw L_J$.
 \end{Exa}

 Given an element $g$ in a component $G^1$ of a possibly disconnected linear algebraic group $G$, we say that $g$ is \emph{quasisemisimple} if there exists a Borel subgroup $B$ and a maximal torus $T\subset B$ in the identity component $G^0$ of $G$ such that $gBg^{-1} = B$ and $gTg^{-1} = T$. Every semisimple element is quasisemisimple \cite[Sections 7.5 and 7.6]{Stein68}. In what follows we will call an element of $wL_J$ quasisemisimple if it quasisemisimple in $N_G(L_J)$.

  \begin{lemma}[{\cite[Section 1.2]{LusztigDisconnectedI}}]
  \label{lem:T1T_map_surjective}
  The map
  \begin{align*}
      T_1\times T& \to T\\
      (t_1, t) & \mapsto tt_1\dw t^{-1}\dw^{-1}
  \end{align*}
  is surjective.
  \end{lemma}
  The following proposition gives properties about the groups $T_1$, $\N_{L_J}$ and $\N_T$.

 \begin{proposition}[{\cite[Section 1.14]{LusztigDisconnectedI}}]
 \label{prop:N_properties }
 The following properties hold.
 \begin{enumerate}
     \item $\N_{L_J}^\circ = T_1$, in particular $\N_{L_{J}}/T_1$ is finite.
     \item Every element of $gT_1$ is quasisemisimple in $N_G(L_J)$.
     \item Every quasisemisimple element of $wL_J$ is conjugated via $L_J$ to an element in $wT_1$.
     \item Two elements $wt_0, wt_1\in wT_1$ are conjugated via $L_J$ if and only if, they are conjugated by an element of $\N_{L_J}$.
     \item $N_{L_J}/N_{T}=W_J^w$.
 \end{enumerate}
 \end{proposition}
 \begin{proof}
    The first four items are in \cite[Section 1.14]{LusztigDisconnectedI}.

 Let us prove item (5). We begin by proving that $\N_{L_J}\subset N_{L_J}(T)$. Take $\ell\in \N_{L_J}$, that is $\ell \dw T_1\ell^{-1}=\dw T_1$. By \cite[Secion 1.14]{LusztigDisconnectedI} we have that $\ell T_1\ell^{-1}=T_1$, and by \cite[Section 1.4(d)]{LusztigDisconnectedI} we get
 \[
 \ell T\ell^{-1}= \ell Z_{L_J}(T_1)\ell^{-1} = Z_{L_J}(\ell T_1\ell^{-1})=Z_{L_J}(T_1)=T.
 \]
 In particular, since $\N_T=\N_{L_J}\cap T$ we have that $\N_T$ is normal in $\N_{L_J}$, and there is a natural injection $\N_{L_J}/\N_T\hookrightarrow \N_{L_J}(T)/T=W_J$. We prove that the image of this map is precisely $W_J^w$. Given $z$ in the image, there exists $\ell \in \N_{L_J}$ such that $\ell \in \dz T$. Since
   \[
   \ell \dw T_1 =\dw T_1 \ell
   \]
   and
   \[
   \ell \dw T_1 \subset B \dz B \dw B = B\dz\dw B \text{ and }  \dw T_1\ell \subset B \dw B \dz B = B\dw\dz B,
   \]
   we must have, by the Bruhat decomposition, that $wz=zw$, i.e., $z\in W_J^w$.

   Conversely, take $z\in W_J^w$. This means that $\dz \dw = \dw \dz t_0$ for some $t_0\in T$ (or, equivalently, $\dz\dw = t_0'\dw\dz$ for $t_0'=\dw\dz t_0\dz^{-1}\dw^{-1}\in T$), in particular if $t\in T_1$, we have
   \[
   \dw \dz t \dz^{-1}\dw^{-1} =t_0'^{-1}\dz\dw t \dw^{-1}\dz^{-1}t_0' = t_0'^{-1}\dz t\dz^{-1}t_0'=  \dz t\dz^{-1},
   \]
   hence $\dz t \dz^{-1}\in T_1$, which means that $\dz T_1\dz^{-1}=T_1$.
   Then we have
   \[
   \dz \dw T_1 \dz^{-1} = \dz \dw \dz^{-1} T_1 = \dz\dw\dz^{-1}\dw^{-1} \dw T_1= t_0' \dw T_1.
   \]
   We claim that upon changing $\dz$ with other representative in $\dz T$ we can assume that $t_0'\in T_1$ and hence $t_0'\dw T_1=\dw t_0'T_1=\dw T_1$. Take $\dz' =t \dz $ with $t\in T$. Then
   \[
   \dz'\dw \dz'^{-1}\dw^{-1}=t \dz \dw \dz^{-1}\dw^{-1} \dw t^{-1} \dw^{-1} = t t_0' \dw t \dw^{-1}.
   \]
   By Lemma \ref{lem:T1T_map_surjective} we have that $t_0'=t^{-1}t_1\dw t \dw^{-1}$ for some $(t_1,t)\in T_1\times T$, then
   $tt_0'\dw t\dw^{-1} = t_1\in T_1$ which concludes the proof.
    \end{proof}

    Define $U^J_w := U^J\cap \dw U^J \dw^{-1} = U^J\cap \dw U \dw^{-1}$. We claim that for every $\ell \in L_J$, we have $\ell U^J_w\ell^{-1} = U^J_w$, and indeed there exists $\ell'\in L_J$ such that $\ell \dw = \dw \ell'$, hence
    \begin{align*}
    \ell U^J_w \ell^{-1} =& \ell U^J\ell^{-1} \cap \ell\dw U^J\dw^{-1}\ell^{-1}\\
     = & U^J\cap \dw \ell'U^J\ell'^{-1}\dw^{-1} \\
     = & U^J\cap \dw U^J\dw^{-1}\\
     = & U^J_w.
    \end{align*}
    The equality $\ell'U^J\ell'^{-1}= U^J$ follows from the fact that $U^J$ is normal in $P_J=L_JU^J$ (see \cite[Proposition 12.6]{MalleTesterman}). This means that $L_JU^J_w$ is a subgroup of $G$.  We have a natural action of $L_JU^J_w$ on $G\times \dw P_J$ given by
   \[
    \ell u\cdot(g,wp) = (gu^{-1}\ell^{-1}, \ell u \dw p u^{-1}\ell^{-1}).
   \]
   This action is well defined, indeed we have
   \[
   \ell u \dw p u ^{-1} \ell^{-1}= \dw (\dw^{-1}\ell \dw) (\dw^{-1}u \dw) p u^{-1}\ell^{-1} \in \dw P_J
   \]
   because $\dw L_J\dw^{-1}=L_J$ and $\dw^{-1}u\dw\in U^J$ by the definition of $U^J_w$.

 \begin{proposition}
 \label{prop:wPJYwJ_iso}
    There is an isomorphism
    \begin{align*}
      f\col \frac{G\times \dw P_J}{L_JU^J_w} &\to \Y_{w,J}^{\circ}\\
      (g, \dw p)& \mapsto (g\dw pg^{-1}, gP_J).
    \end{align*}
   \end{proposition}
 \begin{proof}

   Recall that
    \[
    \Y_{w,J}^\circ = \{(X,gP_J); g^{-1}Xg\in P_JwP_J\}.
    \]
    The map $f$ is $G$-equivariant, where $G$ acts on $G\times \dw P_J$ by left multiplication on $G$ and acts on $Y_{w,J}^\circ$ by conjugation on $X$ and left multpilication on $gP_J$. Let us prove that $f$ is surjective. By the above observation we can assume that $g\in P_J$ and $X\in P_JwP_J$. Therefore, $X=p_0\dw p_1$, hence
    \[
    (X,P_J)=f(p_0, \dw p_1p_0).
    \]
    Now we prove injectivity. Assume that $f(g_1,wp_1)=f(g_2,wp_2)$. By the fact that $f$ is $G$-equivariant, we can assume that $g_1=1$. Hence $g_2\in P_J$ and $g_2\dw p_2g_2^{-1}= \dw p_1$. Since $g_2\in P_J = U^w L_JU^J_w$ and $L_JU^J_w$ acts on $G\times \dw P_J$, we can assume that $g_2\in U^w$. But now we have $g_2\dw (p_2g_2^{-1}) = \dw p_1 \in P_J\dw P_J= U^w\dw P_J$. On the other hand, every element of $U^w\dw P_J$ is uniquely written as $u\dw p$ with $u\in U^w$ and $p\in P_J$. Hence $g_2=1$ and we are done. \par
 \end{proof}

  Since $\N_T\subset T$, we have that $\N_TU^J_w$ is a subgroup of $G$. We have an action of $\N_TU^J_w$ on $G\times \dw T_1U^J$ given by
   \[
    tu\cdot (g, \dw t_0u_0) = (gu^{-1}t^{-1}, tu \dw t_0u_0 u^{-1}t^{-1}).
   \]
 This action is well-defined, because
 \[
 tu \dw t_0u_0 u^{-1}t^{-1} = t \dw t_0 t^{-1} (tt_0^{-1}\dw^{-1} u\dw t_0 t^{-1})(tu_0u^{-1}t^{-1}),
 \]
 and
 \begin{enumerate}
     \item $t\dw t_0t^{-1}\in \dw T_1$,, by the definition of $\N_T$,
     \item $\dw^{-1}u\dw \in U^J$, by the definition of $U^J_w$, and $tt_0^{-1}U^Jt_0t^{-1}=U^J$ because $tt_0^{-1}\in T$,
     \item $tu_0u^{-1}t^{-1}\in tU^Jt^{-1} = U^J$.
 \end{enumerate}

 We also define $(wT_1)^{rs}$ and $(\Y_w^{\circ})^{rs}$ as
\begin{align*}
(\dw T_1)^{rs} &= \{ \dw t \in \dw T_1; \dw t \text{ is regular semisimple}\},    \\
(\Y_w^{\circ})^{w_{rs}} & = \bigg\{ (X, gB); \begin{array}{l}
\text{For every $t\in T$, $u_0,u_1\in U$ such that $g^{-1}Xg = u_0\dw tu_1$,}\\
\text{we have that $\dw t$ is regular semisimple}
\end{array}
\bigg\}.
\end{align*}

\begin{example}
  Keeping the notation from Example \ref{exa:G24_not_galois}, let us prove that the single pair $(X,V_\bullet=V_1\subset V_2\subset V_3)$ in the preimage of $(X,V_2)$ is not in $(\h_{w}^\circ)^{w_{rs}}$. Indeed, we have that $V_1=\langle v_2\rangle $, $V_2=\langle v_2,v_1\rangle$ and $V_3=\langle v_2,v_1,v_4\rangle$. So, we have that $V_\bullet$ is associated to the coset $gB$ for $g=(v_2,v_1,v_4,v_3)$. A computation shows that
  \[
  g^{-1}Xg = \dw
\left(\begin{array}{rrrr}
1 & 0 & 3 & -18 \\
0 & 1 & 0 & -3 \\
0 & 0 & -2 & 24 \\
0 & 0 & 0 & -2
\end{array}\right)
  \]
  and hence
  \[
  \dw t = \left(\begin{array}{rrrr}
0 & 0 & -2 & 0 \\
0 & 0 & 0 & -2 \\
1 & 0 & 0 & 0 \\
0 & 1 & 0 & 0
\end{array}\right)
  \]
  which has characteristic polynomial $(x^2+2)^2$.
\end{example}

The following lemma proves that the condition defining $(\h_w^\circ)^{w_{rs}}$ is independent of the choice of representative in $gB$ and of $u_0$ and $u_1$.\par

\begin{lemma}
 Let $g, X \in G$, $t\in T$ and $u_0,u_1\in U$ be such that $g^{-1}Xg = u_0\dw t u_1$ and $\dw t$ is regular semisimple. If $g'\in gB$, $t'\in T$ and $u_0',u_1'\in U$ are such that $g'^{-1}Xg' = u_0'\dw t' u_1'$, then $\dw t'$ is regular semisimple.
\end{lemma}
\begin{proof}
First, we note that if $g^{-1}Xg=u_0\dw tu_1$, then $t$ depends only on $g^{-1}Xg$ and not on the choice of $u_0, u_1$. \par
Writing $g'=gt_0u$, we have that $u_0'\dw t'u_1' = g'^{-1}Xg' = u^{-1}t_0^{-1}(u_0\dw tu_1)t_0 u$. By the observation above, we can assume that $u=1$, as the conjugation by $u$ will not change the value of $t'$. We can write
\begin{align*}
    t_0u_0\dw tu_1t_0^{-1} = (t_0u_0t_0^{-1})(t_0\dw t t_0^{-1})(t_0u_1t_0^{-1}).
\end{align*}
Since $t_0u_0t_0^{-1}, t_0u_1t_0^{-1}\in U$, and
\[
t_0\dw t t_0^{-1} = \dw (\dw^{-1}t_0\dw)tt_0^{-1}\in \dw T,
\]
we have that $\dw t'=t_0\dw t t_0^{-1}$ which is regular semisimple.
\end{proof}


Define the map $f$ as
   \begin{align*}
     f\colon\frac{G \times \dw T_1U^J}{\N_TU^J_w}& \to Y_w^{\circ}\\
       (g, \dw tu)&\mapsto (g\dw tu g^{-1}, gB).
 \end{align*}
 The map $f$ is well-defined, because $\N_TU^J_w\subset B$ and $\dw tu\in B\dw B$.
 \begin{proposition}
 \label{prop:wT1Ywiso}
   The map above restricts to an isomorpism
    \begin{align*}
      f\col \frac{G \times (\dw T_1)^{rs} U^J}{\N_TU^J_w}& \to (Y_w^{\circ})^{w_{rs}}\\
       (g, \dw tu)&\mapsto (g\dw tu g^{-1}, gB).
   \end{align*}
 \end{proposition}
\begin{proof}

 Let us prove that the map $f$ is surjective. Take $(X, gB)\in (\Y_w^\circ)^{w_{rs}}$. Since $f$ is $G$-equivariant, we can assume that $g=1$. This means that $X\in B\dw B$ and therefore we can write $X=b_0 \dw b_1$. In fact, after conjugation with $b_0^{-1}$ we can assume that $X = \dw t_0 u_0$, with $t_0\in T$, $u_0\in U$ and $\dw t_0$ is regular semisimple.

 By Lemma \ref{lem:T1T_map_surjective} we have that there exists $t\in T$ such that $t\dw t_0t^{-1}= \dw t_1\in \dw T_1$. Upon conjugation by $t$ (and changing $u_0$ appropriately), we can assume that $X=\dw t_0 u_0$ with $t_0\in T_1$. Take $u\in U_J$ and write $u_0 = u_0' u_0''\in U_JU^J $, then
 \[
   uXu^{-1}= \dw t_0 (t_0^{-1}\dw^{-1} u \dw t_0 u_0'u^{-1})  (u u_0'' u^{-1}).
 \]
 Note that $\dw U_J\dw^{-1}=U_J$ (because $wJw^{-1} = J$ and $w\in {}^JW^J)$ and $u u_0'' u^{-1}\in U^J$ (because $U^J$ is normal in $P_J$ by \cite[Proposition 12.6]{MalleTesterman}). So, all we have to do is to prove that there exists $u\in U_J$ such that $t_0^{-1}\dw^{-1} u \dw t_0 u_0'u^{-1}=1$, or equivalently, $u_0' =t_0^{-1}\dw^{-1}u^{-1}\dw t_0u  $.
 We will prove that the following map
 \begin{equation}
 \label{eq:UJUJiso}
 \begin{aligned}
     U_J& \to U_J\\
     u& \mapsto t_0^{-1}\dw^{-1}u\dw t_0u
 \end{aligned}
 \end{equation}
 is an isomorphism. For this it is sufficient to prove injectivity, that is, $Z_{G}(\dw t_0)\cap U_J = \{1\}$. Since $\dw t_0$ is regular semisimple, $Z_{G}(\dw t_0)$ consists of semisimple elements (see \cite[2.11]{Stein65}), and as such, only intersects $U_J$ at the identity. Since the map in Equation \eqref{eq:UJUJiso} is an isomorphism, that means that there exists $u\in U_J$ such that  $u_0' =t_0^{-1}\dw^{-1}u^{-1}\dw t_0u $, so $f$ is surjective.


Let us prove that $f$ is injective. Since $f$ is equivariant, it is enough to prove that if  $(1,\dw t_1 u_1)$ and $(g_2, \dw t_2 u_2)$ are such that $f(1, \dw t_1 u_1)=f(g_2, \dw t_2 u_2)$ then $(1, \dw t_1u_1)=(g_2,\dw t_2u_2)$ in $\frac{G \times (\dw T_1)^{rs} U^J}{\N_TU^J_w}$. The equality $f(1, \dw t_1u_1)=f(g_2, \dw t_2u_2)$ means that $g_2\in B$ and
\[
\dw t_1 u_1 = g_2\dw t_2 u_2 g_2^{-1}.
\]
Since $g_2\in B = TU_JU^wU^J_w$, and we are modding out by the action of $U^J_w$, we can assume that $g_2\in TU_JU^w = U^wTU_J$. We write $g_2= v_1tv_2$ with $t\in T$, $v_1\in U^w$ and $v_2\in U_J$. Then
\begin{align*}
 \dw t_1u_1= g_2\dw t_2 u_2 g_2^{-1} &= v_1tv_2\dw t_2u_2v_2^{-1}t^{-1}v_1^{-1}\\
  &= v_1 \dw (\dw^{-1}t\dw )(\dw^{-1}v_2\dw )t_2u_2v_2^{-1}t^{-1}v_1^{-1}.
 \end{align*}
 Since $\dw^{-1} t\dw \in T$ and $\dw^{-1}v_2\dw\in \dw^{-1}U_J\dw = U_J$, we have, by the uniqueness of the decomposition $B\dw B=U^w\dw B$, that $v_1=1$. So $g_2=tv_2$ and hence
 \begin{align*}
     \dw t_1u_1= g_2\dw t_2 u_2 g_2^{-1} &= tv_2\dw t_2u_2v_2^{-1}t^{-1}\\
                        &=t\dw t_2 t^{-1} (tt_2^{-1}\dw^{-1}v_2\dw t_2t^{-1})(tu_2v_2^{-1}t^{-1}).
 \end{align*}
Since $(tt_2^{-1}\dw^{-1}v_2\dw t_2t^{-1})(tu_2v_2^{-1}t^{-1})\in U$ we must have, by the decomposition $B=TU$, that
\begin{align*}
    \dw t_1 &= t\dw t_2t^{-1},\\
    u_1 &= (tt_2^{-1}\dw^{-1}v_2\dw t_2t^{-1})(tu_2v_2^{-1}t^{-1}).
\end{align*}
From the first equation and Proposition \ref{prop:N_properties } item (4) (applied to case where $J=\emptyset$, that is $L_J=T$) we have that there exists $t'\in \N_T$ such that $\dw t_1 = t'\dw t_2 t'^{-1}$. Upon acting with $t'$ in $(g_2, \dw t_2 u_2)$, we can assume that $t_2 = t_1$. Then the first equation above is equivalent to $t^{-1}t_1\dw = \dw t^{-1}t_1$, which means that $t^{-1}t_1\in T^w$ from which we have $t^{-1}\in T^w \subset \mathcal{N}_T$. Again acting with $t$ we can assume that $g_2= tv_2t^{-1}\in U_J$. We will write $g_2=v_2$. All that is left to prove is that $v_2=1$. We have
\[
\dw t_2 u_1 = v_2\dw t_2u_2v_2^{-1}\text{ and } v_2\in U_J, u_1,u_2\in U^J,
\]
hence
\[
u_1 v_2u_2^{-1}v_2^{-1} = t_2^{-1}\dw^{-1} v_2\dw t_2v_2^{-1}.
\]
However, the left hand side is in $U^J$, because $u_1\in U^J$ and $v_2u_2^{-1}v_2^{-1}\in U^J$ (by the fact that $U^J$ is normal in $P_J$), and the right hand side is in $U_J$. So, both are equal to $1$. By the fact that the map in Equation \eqref{eq:UJUJiso} is an isomorphism, we have that $v_2=1$ and we are done.
\end{proof}
\begin{Rem}
Note that the main ingredient in the proof of Proposition \ref{prop:wT1Ywiso} is the fact that the map in Equation \eqref{eq:UJUJiso} is an isomorphism. So it is possible that the isomorphism locus of the map
\[
\frac{G \times \dw T_1U^J}{\N_TU^J_w} \to Y_w^{\circ}
\]
is larger.
\end{Rem}

\begin{proposition}
\label{prop:non_empty}
When $G=GL_n, SL_n$ we have that $(\dw T_1)^{rs}$ is nonempty.
\end{proposition}
\begin{proof}
Let $w=\tau_1\ldots \tau_k$ the cycle decomposition of $w$. Then $T_1$ is the locus of diagonal matrices $t$ such $t_{ii}=t_{jj}=a_{\tau_\ell}$ whenever $i,j$ belongs to the same cicle $\tau_\ell$. Hence the eigenvalus of $t$ are $a_{\tau_{\ell}}\xi$ for a $|\tau_{\ell}|$-root of unity $\xi$. For a generic choice of $a_{\tau_{\ell}}$ these are all distinct. A similar argument holds for $SL_n$.
\end{proof}
\begin{proposition}
\label{prop:wT1_wLJ}
 We have an isomorphism
   \begin{align*}
       f\col\frac{G\times (\dw T_1)U^J}{\N_{L_J}U^J_w}&\to \frac{G\times (\dw L_J)^{qs}U^J}{L_JU^J_w}\\
        (g, \dw t u)& \mapsto (g, \dw tu).
   \end{align*}
   Moreover, this isomorphism restricts to an isomorphism
   \[
   \frac{G\times (\dw T_1)^{rs}U^J}{\N_{L_J}U^J_w}\to \frac{G\times (\dw L_J)^{rs}U^J}{L_JU^J_w}.
   \]
 \end{proposition}
\begin{proof}
  The surjectivity of $f$ follows from Proposition \ref{prop:N_properties } item (3), while the injectivity follows from Proposition \ref{prop:N_properties } item (4). The isomorphism on the regular semisimple locus follows from the fact that the property of being regular semisimple is invariant under conjugation.
\end{proof}

\begin{proposition}
\label{prop:NLT_galois}
 The map
 \begin{align*}
 \frac{G\times wT_1U^J}{\N_TU^J_w}\to \frac{G\times wT_1U^J}{\N_{L_J}U^J_w}
\end{align*}
is a $\N_{L_J}/\N_T$ Galois cover.
\end{proposition}
\begin{proof}
 We have a free action of $\N_{L_J}/\N_T$ on  $\frac{G\times wT_1U^J}{\N_TU^J_w}$ and
 \[
 \frac{G\times wT_1U^J}{\N_{L_J}U^J_w} = \frac{G\times wT_1U^J}{\N_TU^J_w}\Bigg / \frac{\N_{L_J}}{\N_T}.
 \]
 \end{proof}

We are ready to prove Theorem \ref{thm:pi1_WJw}.

\begin{proof}[Proof of Theorem \ref{thm:pi1_WJw}]
By Proposition \ref{prop:wPJYwJ_iso}, we have that $\h_{w.J}^\circ$ is isomorphic to $\frac{G\times \dw P_J}{L_JU^J_w} = \frac{G\times \dw L_JU^J}{L_JU^J_w}$. We define $\U_{w,J}$ to be the open subset of $\h_{w,J}^\circ$ given by $\frac{G\times (\dw L_J)^{rs}U^J}{L_JU^J_w}$ (which is nonempty by Propositions \ref{prop:non_empty} and \ref{prop:wT1_wLJ}). By Proposition \ref{prop:NLT_galois} we have that the following map is a $\N_{L_J}/\N_T=W_J^w$ Galois cover
\[
\frac{G\times (wT_1)^{rs}U^J}{\N_TU^J_w}\to \frac{G\times (wT_1)^{rs}U^J}{\N_{L_J}U^J_w}= \U_{w,J}.
\]
This gives a map $\pi_1(\U_{w,J},(X,gP_J))\to W_J^w$. \par
To prove the last statement, we just use Proposition \ref{prop:wT1Ywiso} to see that $\frac{G\times (wT_1)^{rs}U^J}{\N_TU^J_w}$ is precisely $f^{-1}(\U_{w,J})$ (recall that $f$ is the forgetful map $f\colon \h_w\to \h_{w,J}$).
\end{proof}

\section{Classification}
\label{sec:classification}

This section is devoted to prove that when $G=GL_n$ the perverse sheaves in Theorem \ref{thm:pi1_WJw} are the only perverse sheaves appearing in Equation \eqref{eq:decomposition_hw}, that is, every $1$-character sheaf with full support on $\h_{w,J}$ comes from a representation of $W_{J_{w}}^w$. From now on, we restrict ourselves to the case $G=GL_n$, $W=S_n$. 
Let $J$ be a subset of the set of simple transpositions of $S_n$ and fix $w\in {}^JW^J$ such that $wJw^{-1}=J$. Let $\lambda=(\lambda_1,\ldots, \lambda_{\ell(\lambda)})$ be the composition such that
\[
W_J=S_{\lambda_1}\times S_{\lambda_2}\times\ldots S_{\lambda_{\ell(\lambda)}} \text{ and }L_J = \prod_{i=1}^{\ell(\lambda)} GL_{\lambda_i}.
\]

Since $wJw^{-1}=J$ and $w\in {}^JW^J$, $w$ acts on $W_J$ and $L_J$ by permuting coordinates, that is, there exists a permutation $\sigma\col [\ell(\lambda)]\to[\ell(\lambda)]$ such that
\[
w^{-1}\cdot (z_1,\ldots, z_{\ell(\lambda}) = (z_{\sigma(1)},\ldots, z_{\sigma(\ell(\lambda))})
\]
for every $(z_1,\ldots, z_{\ell(\lambda)})\in W_J$ and
\[
w^{-1}\cdot (g_1,\ldots, g_m) = (g_{\sigma(1)},\ldots, g_{\sigma(\ell(\lambda)})
\]
for every $(g_1,\ldots, g_{\ell(\lambda)})\in L_J$. In particular $\lambda_{\sigma(i)}=\lambda_i$ for every $i\in [\ell(\lambda)]$.

\begin{proposition}
\label{prop:WJw_product}
Let $\sigma=\tau_1\ldots \tau_k$ be the cycle decomposition of $\sigma$,  then  $W_J^w$ is isomorphic to $S_{\lambda_{i_1}}\times \ldots \times S_{\lambda_{i_k}}$, where $i_j\in[\ell(\lambda)]$ is an element not fixed by $\tau_j$.
\end{proposition}
\begin{proof}
We abuse notation and write $i\in \tau_j$ if $i$ is not fixed by $\tau_j$. We have to find $(z_1,\ldots, z_{\ell(\lambda)})\in W_J$ such that $z_i=z_{\sigma(i)}$ for every $i\in [\ell(\lambda)]$. This is equivalent to $z_i=z_{i'}$ whenever $i,i'$ belongs to the same $\tau_j$. Choosing $i_j\in \tau_j$, we have that
\[
W_J^w=\{(z_1,\ldots, z_{\ell(\lambda)}); z_i=z_{i_j}\text{ for every $i,j$ with }i\in \tau_j\},
\]
and the result follows.
\end{proof}

For $z \in W_J$ we define (recall Example \ref{exa:character_sheaves_LJ})
\begin{align*}
{}^w\h_{L_J,z}^{\circ}:=&\{(\ell, \ell_0B_J); \ell,\ell_0\in L_J; \dw^{-1}\ell_0^{-1}\dw \ell \ell_0\in B_J\dz B_J\}, \\
{}^w\h_{P_J,z}^{\circ}:=&\{(X, gB); X\in \dw P_J,g\in P_J; g^{-1} X g \in B\dw\dz B\}.
\end{align*}
\begin{proposition}
  \label{prop:PJLJ_bundle}
  We have a natural isomorphism ${}^w\h_{P_J,z}^{\circ}\to {}^w\h_{L_J,z}^{\circ}\times U^J$.
\end{proposition}
\begin{proof}
  Define the map $f\colon {}^w\h_{P_J,z}^{\circ}\to {}^w\h_{L_J,z}^{\circ}\times U^J$ as follows: Given a point $(X,gB)\in {}^w\h_{P_J,z}^{\circ}$, we can write $X=\dw \ell u$ and $g = \ell_0 u_0$, and set $f(X, gB)=((\ell, \ell_0B_J), u)$. We claim that this $f$ is well-defined. First note that if $g'=\ell_0'u_0'$ is such that $g'B=gB$ then $\ell_0'B=\ell_0$, and hence $\ell_0'B_J=\ell_0B_J$. Next, we have to show that if $\ell_0^{-1}\dw \ell u \ell_0 \in B\dw\dz B$, then $\dw^{-1}\ell_0^{-1}\dw \ell \ell_0\in B_J\dz B_J$. For this, note that $\ell_0^{-1}\dw \ell u \ell_0 \in B\dw\dz B$ is equivalent to $\ell_0^{-1}\dw \ell  \ell_0 \in B\dw\dz B$ because $\ell_0^{-1}u\ell_0\in U^J\subset B$. On the other hand $\dw^{-1}\ell_0^{-1}\dw\ell\ell_0\in L_J$, and by the Bruhat decomposition there is a unique $z'\in W_J$ such that $\dw^{-1}\ell_0^{-1}\dw\ell\ell_0\in B_J\dz' B_J$. Thus $\ell_0^{-1}\dw\ell\ell_0 \in \dw B_J\dz' B_J\subset B\dw B\dz' B= B\dw\dz' B$, so $z=z'$, hence the map is well-defined. Injectivity and surjectivity are clear.

\end{proof}

\begin{proposition}
\label{prop:toLJ_small}
The map ${}^w\h_{L_J,e}\to L_J$ given by the projection onto the first factor is small.
\end{proposition}
\begin{proof}
 Since $w^{-1}Jw = J$ and $w\in {}^JW^J$ we have that $W_J=S_{\lambda_1}\times S_{\lambda_2}\times\ldots S_{\lambda_{\ell(\lambda)}}$ and $L_J = \prod_{i=1}^{\ell(\lambda)} GL_{\lambda_i}$. Moreover, $w$ acts on $L_J$ by permuting coordinates, so there exists a permutation $\sigma\col [\ell(\lambda)]\to[\ell(\lambda)]$ such that $w^{-1}\cdot (g_1,\ldots, g_m) = (g_{\sigma(1)},\ldots, g_{\sigma(m)})$, and in particular $\lambda_{\sigma(i)}=\lambda_i$. Assume first that $\sigma = (12\ldots m)$ is a $m$-cycle (in particular $\lambda_i=\lambda_{i'} = k$ for every $i,i'\in[\ell(\lambda)]$). Then we have a fiber diagram
 \[
 \begin{tikzcd}
    {}^w\h_{L_J,e} =\{(\ell, gB_J); \ell,g\in L_J; \dw^{-1}g\dw \ell g\in B_J\} \ar[r, "f"] \ar[d, "p_2"]& L_J=GL_k\times GL_k\times\cdots \times GL_k \ar[d, "p_1"]\\
    \{(X, g_1B_k); g_1^{-1}Xg_1\in B_k\}\ar[r, "g"] & GL_k
 \end{tikzcd}
 \]
  where $p_1(\ell_1,\ldots, \ell_m)=\ell_m\ell_{m-1}\cdots \ell_1$ and
  \[
   p_2((\ell_1,\ldots, \ell_m), (g_1,\ldots, g_m)B_J)=(\ell_m\cdots \ell_1, g_1).
  \]
  The map $p_2$ is well defined since the condition $\dw^{-1}g\dw\ell g\in B_J$ is equivalent to
  \[
  (g_2^{-1}\ell_1g_1, g_3^{-1}\ell_2g_2,\ldots, g_1^{-1}\ell_mg_m)\in B_k\times\cdots\times B_k.
  \]
  Hence, $g_2B_k,\ldots, g_mB_k$ are uniquely determined by $(\ell_1,\ldots, \ell_m)$ and $g_1$. Moreover, the condition on $g_1$ becomes
  \[
   g_1^{-1}\ell_m\ell_{m-1}\cdots \ell_1 g_1\in B_k.
  \]
  The maps $p_1$ and $p_2$ are $GL_k\times \cdots\times GL_k$-bundles (the product has $\ell(\lambda)-1$ factors) and hence are flat. Since $g$ is small by Proposition \ref{prop:Grothe_Springer_Small}, $f$ is small. The general case when $\sigma$ is not a cycle follows from the cycle decomposition of $\sigma$ since the product of small maps is small, that is, if $f_i\col Y_i\to X_i$ are small maps, where $i$ runs through elements of a finite set $I$, then $(f_i)_{i\in I}\col \prod_{i\in I} Y_i\to \prod_{i\in I} X_i$ is small.
\end{proof}

\begin{proposition}
\label{prop:towPJ_small}
The map ${}^w\h_{P_J,e}\to wP_J$ is small.
\end{proposition}
\begin{proof}
 Consider the following commutative diagram
 \[
 \begin{tikzcd}
   {}^w\h_{P_J,e} \ar[r]\ar[d, "p_2"] &  \dw P_J \ar[d, "p_1"]\\
   {}^w\h_{L_J,e} \ar[r] & L_J
 \end{tikzcd}
 \]
 where the map $p_2$ is the projection induced by Proposition \ref{prop:PJLJ_bundle} and the map $p_1$ is given by $p_1(\dw \ell u) = \ell$ (recall that $P_J=L_JU^J$). The diagram is actually fibered, where the fiber of $p_2$ over a point $(\ell, \ell_0B_J)$ is $\{(\dw \ell u, \ell_0B), u\in U^J\}$, mapping isomorphically (via ${}^w\h_{P_J,e}\to \dw P_J$) to $\dw \ell U^J \subset \dw P_J$, which is the fiber of $p_1$ over the point $\ell \in L_J$.  Since $p_1$ and $p_2$ are both $U^J$-bundles, the vertical maps are flat, and since the bottom map is small,  Proposition \ref{prop:toLJ_small} implies the result.


\end{proof}

\begin{proposition}
\label{prop:char_LJ}
  Given $z\in W_J$, let $f_z\col{}^w\h_{L_J,z}^\circ\to L_J$ be the  projection onto the first factor. Then every simple summand of $(f_z)_!(\mathbb{C}_{{}^w\h_{L_J,z}^{\circ}})$ is also a summand of $(f_e)_!(\mathbb{C}_{{}^w\h_{L_J,e}})$. The same holds for $\overline{f}_z\col {}^w\h_{L_J,z}\to L_J$.
\end{proposition}
\begin{proof}
As before we can assume that  $L_J=GL_k\times \cdots\times GL_k$, that $W_J=S_k\times \cdots \times S_k$, and that $\dw^{-1}\cdot(g_1,\ldots, g_m)=(g_2,\ldots, g_1)$. The general case follows from the fact that if $f_i\col Y_i\to X_i$ are continuous maps indexed by a finite set $i \in I$, then
 \[
 ((f_i)_{i\in I})_!(\mathbb{C}_{\prod_{i\in I}Y_i}) = \boxtimes_{i\in I} (f_i)_!(\mathbb{C}_{Y_i}).
 \]

Writing $z=(z_1,\ldots, z_m)\in W_J=S_k\times\cdots\times S_k$, we have a fiber diagram
 \[
 \begin{tikzcd}
    {}^w\h_{L_J,z}^\circ \ar[r, "f_z"] \ar[d, "p_2"]& L_J=GL_k\times GL_k\times\cdots \times GL_k \ar[d, "p_1"]\\
    \h_{GL_k,z}^{\circ}\ar[r, "g_z"] & GL_k
 \end{tikzcd}
 \]
 where (recall Equation \eqref{eq:Yw})
 \[
   \h_{GL_k,z}^{\circ} := \{(X, g_1B_k, g_2B_k,\ldots, g_{m+1}B_k); g_1^{-1}Xg_{m+1}\in B_k; g_{i+1}^{-1}g_i\in B_k\dz_i B_k\}
 \]
 and the map $p_2$ is given by
 \begin{multline*}
  p_2((\ell_1,\ldots, \ell_m),(g_1B_k,\ldots, g_mB_k))=\\
  =(\ell_m\cdots \ell_1, g_1B_k, \ell_1^{-1}g_2B_k, \ell_1^{-1}\ell_2^{-1}g_3B_k, \ldots, \ell_1^{-1}\cdots \ell_{m-1}^{-1}g_mB_k, \ell_1^{-1}\cdots \ell_m^{-1}g_1B_k).
 \end{multline*}
  By \ref{prop:char_sheaves_GLn}, every simple summand of $(g_z)_!(\mathbb{C}_{\h_{GL_k,z}^{\circ}})$ is a summand of $(g_e)_!(\mathbb{C}_{\h_{GL_k,e}})$. In particular, since the diagram is fibered and the maps $p_1$ and $p_2$ are smooth, every simple summand of $(f_z)_!(\mathbb{C}_{{}^w\h_{L_J,z}^\circ})$ is a summand of $(f_e)_*(\mathbb{C}_{{}^w\h_{L_J,e}^{\circ}})$.\par
   The proof for $\overline{f}_z$ is identical, simply swapping $p_2$ for $\overline{p}_2\colon {}^w\h_{L_J,z}\to \h_{GL_k,z}$.
\end{proof}

\begin{proposition}
We have a natural action of $L_JU^J_w$ on $\dw P_J$ given by conjugation. If we identify $P_J=L_JU^J$ then this action is given by
\begin{align*}
    (\ell,u)\cdot (\ell_0,u_0) = (\dw^{-1}\ell\dw \ell_0\ell^{-1}, \ell \ell_0^{-1}\dw^{-1}u\dw\ell_0u_0u^{-1}\ell^{-1}).
\end{align*}
\end{proposition}
\begin{proof}
It is enough to check that
\[
\ell u \dw \ell_0u_0 u^{-1}\ell^{-1} = \dw (\dw^{-1}\ell \dw \ell_0\ell^{-1})(\ell \ell_0^{-1}\dw^{-1}u\dw\ell_0u_0u^{-1}\ell^{-1}).
\]
Also note that
\[
\ell \ell_0^{-1}\dw^{-1}u\dw\ell_0u_0u^{-1}\ell^{-1} = \ell ((\ell_0^{-1}(\dw^{-1}u\dw)\ell_0)u_0u^{-1})\ell^{-1}
\]
is actually in $U^J$ as each expression between two parenthesis is in $U^J$, either because $\dw^{-1}U^J_w\dw\subset U^J$ or because $\ell U^J\ell^{-1} = U^J$ for every $\ell \in L_J$ ($U^J$ is normal in $P_J$).
\end{proof}

\begin{corollary}
  We have a natural action of $L_JU^J_w$ on ${}^w\h_{P_J,z}$ given by
  \[
  (\ell u)\cdot (X, gB) = ((\ell u)X(\ell u)^{-1}, \ell ugB).
  \]
\end{corollary}

\begin{proposition}
\label{prop:GLLUY}
We have a natural isomorphism
\begin{align*}
    \frac{G\times ({}^w\h_{P_J,z}^{\circ})}{L_JU^J_w}  &  \to \h_{wz}^\circ\\
      (g_0, (X,gB)) &\mapsto (g_0Xg_0^{-1}, g_0gB).
\end{align*}
\end{proposition}
\begin{proof}
We begin by noticing that the map is well defined. Indeed $(g_0g)^{-1}g_0Xg_0^{-1}g_0g = g^{-1}Xg\in B\dw\dz B$ and the image of $(g_0(\ell u)^{-1}, (\ell u)X(\ell u)^{-1}, \ell u gB)$ is the same as the image of $(g_0, (X,gB))$. Every pair $(X,gB)\in \h_{wz}^{\circ}$ satisfies $g^{-1}Xg\in B\dw\dz B$, and after changing the representative $g$ in $gB$, we may assume that $g^{-1}Xg\in \dw\dz B \subset \dw P_J$, hence $(X,gB)$ is the image of $(g, (g^{-1}Xg, B))$. This proves surjectivity. Given two pairs $(g_0, (X, gB))$ and $(g_0', (X', g'B))$ with the same image, we can assume $g_0'=1$. Then $g_0^{-1}Xg_0 = X'$ and $g_0gB=g'B$. Since $g,g'\in P_J$ we have that $g_0\in P_J$ as well. Then we can write $g_0=\ell_0u_0u_0' \in L_JU^{J}_wU^w = P_J$ and $X=\dw p\in \dw P_J$. Hence
\begin{align*}
    X' & = g_0^{-1}Xg_0\\
    & = u_0'^{-1}u_0^{-1}\ell_0^{-1}\dw p g_0\\
                & = u_0'^{-1}\dw (\dw^{-1}u_0^{-1}\dw)(\dw^{-1}\ell_0^{-1}\dw) p g_0\in u_0'^{-1}\dw P_J.
\end{align*}
Since $X'\in \dw P_J$ and  $P_J\dw P_J = U^w\dw P_J$, we conclude that $u_0'=1$, and hence $g_0\in L_JU^J_w$, so $(g_0,(X,gB))$ and $(1,(X', g'B))$ are the same point in $\frac{G\times ({}^w\h_{P_J,z}^{\circ})}{L_JU^J_w}$.
\end{proof}

\begin{proposition}
\label{prop:YwYwJ_pushforward}
Let $f\col \h_{w}^\circ \to \h_{w,J}^\circ$ be the natural forgetful map. Consider $L$ the local system on $\U_{w,J}$ corresponding to the permutation representation of $W_J^w$ on itself. Then $f_*(\mathbb{C}_{\h_w^\circ})$ is the intermediate extension of $L$ to $\h_{w,J}^\circ$.
\end{proposition}
\begin{proof}
We begin by proving that $f$ is small. In fact, it is enough to prove that
\[
\frac{G\times {}^{w}\h_{P_J,e}}{L_JU^J_w}\to \frac{G\times \dw P_J}{L_JU^J_w}
\]
is small. This follows from Proposition \ref{prop:towPJ_small} since $L_JU^J_w$ acts freely on both sides. Since $f$ is small, we can use Proposition \ref{prop:small_local_system} to conclude that $f_*(\mathbb{C}_{\h_w^\circ})$ is the intermediate extension of $(f|_{(\h_{w}^\circ)^{w_{rs}}})_*(\mathbb{C}_{(\h_{w}^\circ)^{w_{rs}}})$. Since the map $(\h_{w}^\circ)^{w_{rs}}\to \U_{w,J}$ is a $W_J^w$-Galois covering, the local system $(f|_{(\h_{w}^\circ)^{w_{rs}}})_*(\mathbb{C}_{(\h_{w}^\circ)^{w_{rs}}})$ is precisely the local system that corresponds to the permutation representation of $W_J^w$ on itself.
\end{proof}

\begin{proposition}
\label{prop:wz_to_w}
Let $z\in W_J$ and consider the map $f'\col Y_{wz}^{\circ}\to Y_{w,J}^{\circ}$. Then every simple summand of $f'_!(\mathbb{C}_{\h_{wz}^\circ})$ is induced by an irreducible representation of $W_J^w$.
\end{proposition}
\begin{proof}
We must show that every simple summand of $f'_!(\mathbb{C}_{\h_{wz}^{\circ}})$ is a summand of $f_*(\mathbb{C}_{\h_{w}^{\circ}})$ and use Proposition \ref{prop:YwYwJ_pushforward}. By Proposition \ref{prop:GLLUY}  we have isomorphisms
\begin{align*}
    \frac{G\times ({}^w\h_{P_J,z}^{\circ})}{L_JU^J_w} & \to \h_{wz}^{\circ},\\
    \frac{G\times ({}^w\h_{P_J,e}^{\circ})}{L_JU^J_w} & \to \h_{w}^{\circ}, \\
    \frac{G\times \dw P_J}{L_JU^J_w} & \to \h_{w}^{\circ}.
\end{align*}
By Proposition \ref{prop:smooth_pullback_perverse} it is enough to prove that every simple summand of $\overline{f}'_!(\mathbb{C}_{({}^w\h_{P_J,z}^{\circ})}))$ is a summand of $\overline{f}_*(\mathbb{C}_{({}^w\h_{P_J,e}))})$ where $\overline{f}'$ and $\overline{f}$ are the maps
\begin{align*}
    \overline{f}'\col {}^w\h_{P_J,z}^{\circ}\to P_J,\\
    \overline{f}\col {}^w\h_{P_J,e}\to P_J.
\end{align*}
By propositions \ref{prop:PJLJ_bundle} and \ref{prop:smooth_pullback_perverse} we can reduce even further and consider the maps
\begin{align*}
    \overline{f}'\col {}^w\h_{L_J,z}^{\circ}\to L_J,\\
    \overline{f}\col {}^w\h_{P_J,e}\to L_J.\\
\end{align*}
The result then follows from Proposition \ref{prop:char_LJ}.
\end{proof}


\begin{theorem}
\label{thm:local_system_parabolic}
 Let $w\in S_n$ and $J\subset \{1,\ldots, n-1\}$, and consider the map $f\col \h_w\to G\times G/P_J$. Then
 \[
 f_*(IC_{\h_w})=\bigoplus_{w'\in {}^JW} IC_{\h_{w', J}}(L_{w,w'}),
 \]
 where $L_{w,w'}$ is a local system on $\U_{w',J'}$ induced by a (graded) representation of $W_{J_{w'}}^{w'}$.
\end{theorem}
\begin{proof}


By the results in \cite[Section 4]{LusztigParabolicI}, the $1$-character sheaves on $G\times G/P_J$ are the intermediate extensions of the simple perverse sheaves on $\h_{w',J}^{\circ}$ induced by $1$-character sheaves on $w'L_{J_{w'}}$. In the notation of Lusztig, this just means the classes $\mathcal{C}_{J,\delta}$ and $\mathcal{C}'_{J,\delta}$ coincide. The variety $w'L_{J_{w'}}$ is the connected component $C$ of the normalizer $N_{G}(L_{J_{w'}})$, see \cite[Section 4.6]{LusztigParabolicI}. The 1-character-sheaves on $w'L_{J_{w'}}$ are precisely the summands of
\[
f_{z!}(\mathbb{C}_{{}^{w'}\h_{L',z}}),
\]
where $L'=L_{J_{w'}}$, $z\in W_{J_{w'}}$ and $f_z\colon {}^{w'}\h_{L',z}^{\circ}\to L_{J_{w'}}$ is the projection onto the first factor (see Example \ref{exa:character_sheaves_LJ}). The result follows from Propositions \ref{prop:char_LJ} and \ref{prop:wz_to_w}.
\end{proof}

Theorem \ref{thm:UwJ_intro} is now a direct corollary of Theorem \ref{thm:local_system_parabolic}.

We finish this section with a couple of examples. The first is a computation of $\ch(H^*(\h_{w,J},L))$ where $L$ is not trivial. The second is the interpretation in terms of the chromatic symmetric function of $\ch(\h_{w,J})$ for special pairs $(w,J)$.
\begin{example}
\label{exa:ch_G24}
Fix $w=3412\in S_4$, $J=\{1,3\}$ and $X$ a regular semisimple $4\times 4$ diagonal matrix. Recall that $\h_{w,J}^{rs} = G^{rs}\times \Gr(2,4)$ and $\h_{w,J}(X)=\Gr(2,4)$. Let us compute $\ch(IC_{G\times \Gr(2,4)}(L))$ where $L$ is the local system on $\U_{3412,\{1,3\}}$ induced by the regular representation of $W_J^w=S_2$. Consider the variety
\[
\mathcal{Z}= \left\{\left( X, \begin{array}{c}V_1\subset V_2\subset V_3\\ V_1\subset V_2'\subset V_3
  \end{array}\right);XV_1\subset V_2'\subset XV_3 , X\in G^{rs} \right\}
\]
and the natural map $f\col\mathcal{Z}\to G^{rs}\times \Gr(2,4)$ that keeps $V_2$. The map $f$ is generically finite, and factors through $\mathcal{Z}\to \h_{3412}^{rs}$ (which is a small map that is generically injective). This means that, $f_*(\mathbb{C}_{\mathcal{Z}})=IC_{\h_{w,J}^{rs}}(L) \oplus \mathcal{F}$, where $\mathcal{F}$ is a sum of simple perverse sheaves that do not have full support.

Let us prove that $f$ is semi-small. We will consider each strata of $G^{rs}\times \Gr(2,4)$
\begin{enumerate}
    \item If $(X,V_2)\in (\h_{3412, J}^\circ)^{rs}$ we have that the fiber $f^{-1}(X,V_2)$ is finite. Indeed, we have that $XV_2\cap V_2=\{0\}$ so we can write $V_2=\langle v_1,v_2\rangle$ such that $v_1,v_2,Xv_1,Xv_2$ are linearly indepedent. Let $(X,V_1\subset V_2'\subset V_3)\in f^{-1}(X,V_2)$. Write $V_1=\langle av_1+bv_2\rangle$, since $XV_1\subset V_2'\subset V_3$ and $XV_1\cap V_2=\{0\}$ we have that $V_3=V_2+XV_1$ and $V_2'=V_1+XV_1$. Moreover, we must have that $V_1\subset V_2'\subset XV_3$ so $X^{-1}V_1\subset V_2+XV_1$ which is equivalent to the vaninshing of the determinant of the matrix $(v_1,v_2, aXv_1+bXv_2,aX^{-1}v_1+bX^{-1}v_2)$ which is a non zero homogeneous polynomial of degree 2 in $a$ and $b$. So, we have at most two points in $f^{-1}(X,V_2)$.

    \item If $(X,V_2)\in (\h_{3142,J}^{\circ})^{rs}$ we have that the fiber $f^{-1}(X,V_2)$ is finite. Indeed, we have that $\dim XV_2\cap V_2=1$ so we can write $V_2=<v_1, Xv_1>$ and moreover $v_1, Xv_1, X^2v_1, X^3v_1$ are linearly independent. By repeating the arguments above, we have the vanishing of the determinant of the matrix $(v_1, Xv_1, bX^2v_1, aX^{-1}v_1)$, hence the fiber $f^{-1}(X,V_2)$ is given by $<v_1> \subset V_2'=V_2 \subset <v_1, Xv_1, X^{-1}v_1>$ and $<Xv_1>\subset <Xv_1, X^2v_1>\subset <v_1, Xv_1, X^2v_1>$.

    \item If $(X,V_2)\in (\h_{1342,J}^{\circ})^{rs}$ we have that fiber of $f^{-1}(X,V_2)$ is isomorphic to $\mathbb{P}^1$. Indeed, we have that there exists $v_1\in V_2$ such that $Xv_1=v_1$. For each $V_3\in \mathbb{P}(\mathbb{C}^4/V_2)$ we have that $<v_1>\subset V_3\cap X^{-1}V_3\subset V_3$ belongs to the fiber $f^{-1}(X,V_2)$ (we note that $XV_3\neq V_3$ because $(X,V_2)\notin \h_{1324,J}^\circ$). Let us prove now that if $V_1\subset V_2$ is different from $<v_1>$ then there does not exist any flag $V_1\subset V_2'\subset V_3$ in the fiber. Since $V_1\neq <v_1>$ we have that $XV_1\not\subset V_2$ (otherwise $XV_2=V_2$ and then $(X,V_2)\in \h_{1234,J}^\circ$). Writing $V_1=<v_2>$, we have that $V_2=<v_1,v_2>$ and since $XV_1\subset V_3$, we have $V_3=<v_1,v_2,XV_2>$. On the other hand, we must also have that $X^{-1}V_1\subset V_3$, which implies that $X^{-1}v_2\in <v_1,v_2, Xv_2>$, which means that $X^{-1}V_3=V_3$. However $(X,V_2)\notin \h_{1324,J}^{\circ}$, so there does not exist any $V_3$ containing $V_2$ satisfying $XV_3=V_3$.

    \item If $(X,V_2)\in (\h_{3124,J}^{\circ})^{rs}$ we have that the fiber of $f^{-1}(X,V_2)$ is isomorphic to $\mathbb{P}^1$. This is dual to the case above.

    \item If $(X,V_2)\in (\h_{1324,J}^{\circ})^{rs}$ we have that fiber of $f^{-1}(X,V_2)$ is isomorphic to a chain $\mathbb{P}^1\cup \mathbb{P}^1\cup \mathbb{P}^1$. Indeed, there exists $v_1,v_2,v_3$ with $Xv_i=v_i$ for $i=1,2,3$ and $<v_1>\subset V_2\subset <v_1,v_2,v_3>$. The three $\mathbb{P}^1$ are given as follows.
    \begin{enumerate}
        \item The flag  $<v_1>\subset V_2'\subset <v_1,v_2,v_3>$ is in the fiber for each $V_2'$.
        \item The flag $<v_1>\subset V_3\cap X^{-1}V_3\subset V_3$ is in the fiber for each $V_3\in \mathbb{P}(\mathbb{C}^4/V_2)$.
        \item The flag $V_1\subset V_1+XV_1\subset <v_1,v_2,v_3>$ is in the fiber for each $V_1\in \mathbb{P}(V_2)$.
    \end{enumerate}
    Let us prove that if $V_1\neq <v_1>$ and $V_3\neq <v_1,v_2,v_3>$ there does not exist $V_2'$. Since $XV_2\neq V_2$ and $XV_1\subset V_3$ we must have $V_3=V_2+XV_1$ and $XV_3\neq V_3$ because $V_2\neq <v_1,v_2,v_3>$. The argument now follows as in item (3).

    \item If $(X,V_2)\in (\h_{1234,J}^{\circ})^{rs}$ we have that fiber of $f^{-1}(X,V_2)$ is isomorphic to the blow up of $\mathbb{P}^1\times \mathbb{P}^1$ at the four points
    \begin{equation}
    \label{eq:points}
        ((1:0),(1:0)),\; ((0:1),(1:0)), \;((1:0),(0:1)),\; ((0:1),(0:1)).
    \end{equation}
    Indeed, since $XV_2=V_2$, we have that for every $V_1, V_3$ satisfying $V_1\subset V_2\subset V_3$ there exists $V_2'=V_2$ satisfying $V_1\subset V_2'\subset V_3$ and $XV_1\subset V_2'\subset V_3$. Since $V_1+XV_1\subset V_2'$ and $V_2'\subset V_3\cap XV_3$, we have that $V_2'$ is unique, except in the cases where $XV_1=V_1$ and $XV_3=V_3$. These cases correspond to the four points in Equation \eqref{eq:points} but now, any choice of $V_2'$ such that $V_1\subset V_2'\subset V_3$ will give a point in the fiber.
    \end{enumerate}

    This proves that the map $f$ is semi-small (see \cite{BorhoMacpherson}, \cite[Section 4.2]{CatMig09}) with relevant locus $\h_{1342,J}$, $\h_{3124,J}$ and $\h_{1234,J}$. Since the fibers in these loci are irreducible, we have that
    \[
    f_*(\mathbb{C}_\mathcal{Z})[-4]=IC_{\h_{3412,J}}(L)\oplus IC_{\h_{1342,J}}(\mathbb{C})[-1]\oplus IC_{\h_{3124,J}}(\mathbb{C})[-1]\oplus IC_{\h_{1234,J}}(\mathbb{C})[-2].
    \]

    We conclude that
    \begin{align*}
    \ch(IC_{\Gr(2,4)}(L)) = &(1+q)((q+q^2)h_{2,2}+(q+q^2)h_{3,1}+(1+q+q^2+q^3)h_4)\\
                            & - 2q(1+q+q^2)h_{3,1} - q^2h_{2,2}\\
                          = & (q+q^2+q^3)h_{2,2} - (q+q^3)h_{3,1} + (1+2q+2q^2+2q^3+q^4)h_4\\
                          = & (q+q^2+q^3)s_{2,2} + q^2s_{3,1} + (1+2q+3q^2+2q^3+q^4)s_4.
    \end{align*}

    \noindent Consequently, if $L'$ is the local system induced by the sign representation of $S_2$,
    \begin{align*}
        \ch(IC_{\Gr(2,4)}(L')) = &\ch(IC_{\Gr(2,4)}(L)) - \ch(IC_{\Gr(2,4)}(\mathbb{C}))\\
                                = &(q+q^2+q^3)h_{2,2} - (q+q^3)h_{3,1} + (q+q^3)h_4\\
                                = &(q+q^2+q^3)s_{2,2} + q^2s_{3,1} + (q+q^2+q^3)s_4
    \end{align*}
\end{example}

The following example was stablished in \cite[Section 4]{KiemLee}, where it is referred as  thegeneralized Shareshian--Wachs conjecture.
\begin{example}[Chromatic symmetric function of weighted graphs following Gasharov \cite{Gasharov}]
\label{exa:chromatic_gasharov}
Let $G$ be a graph with vertex set $[m]$ and let $f\col [n]\to [m]$ be a surjective non-decreasing function. A proper $\underline{f}$-coloring of $G$ is a function $\kappa\col [n]\to \mathbb{P}$ such that $\kappa(i)\neq \kappa(j)$ whenever $f(i)$ and $f(j)$ are adjacent in $G$. Moreover, an ascent of $\kappa$ is a pair $i<j$ such that $\kappa(i)<\kappa(j)$ and $f(i)<f(j)$. We define
\[
\csf_q(G,f;x,q) := \sum_{\text{proper } \kappa\col [n]\to \mathbb{P}}q^{\asc_{(G,f)}(\kappa)}\prod_{i=1}^{n}x_{\kappa(i)}.
\]
If $G^f$ is the graph obtained from substituting each vertex $j$ of $G$ with a clique $K_{f^{-1}(j)}$, then we have that
\[
\csf_q(G,f;x,q) = \frac{\csf_q(G^f;x,q)}{\prod_{j=1}^m (|f^{-1}(j)|)!_q}.
\]
Let $\m\col[n]\to[n]$ be a Hessenberg function and $w_\m$ be its associated codominant permutation. Define
\[
J:=\{j\in \{1,\ldots, n-1\}; j\notin \Ima(\m), \m(j)=\m(j+1)\}.
\]
We claim that $w_\m$ is maximal in $W_Jw_\m W_J$. This is equivalent to $w_\m(j+1)<w_\m(j)$ and $w_\m^{-1}(j+1)<w_\m^{-1}(j)$. If $\m(j)=\m(j+1)$, then $w_\m(j+1)<w_\m(j)$. If $j\notin \Ima(\m)$, then  $w_\m^{-1}(j+1)<w_\m^{-1}(j)$. We have
\[
\h_{w_\m}=\{(X,V_\bullet); XV_i\subset V_{\m(i)}, i\in [n]\setminus J\}.
\]
Let $w_0$ be the minimum element in $W_Jw_\m$. By Proposition \ref{prop:PJB_bundle} we have that
\[
\h_{w_0,J} = \{(X,V_\bullet); XV_i\subset V_{\m(i)}, i\in [n]\setminus J\}\subset G\times G/P_J
\]
and that $\h_{w_\m}\to \h_{w_0,J}$ is a $P_J/B$ bundle. Hence,
\[
\ch(H^*(\h_{w_0,J}))=\frac{\ch(H^*(\h_{w_\m}))}{|W_J|}.
\]
This means that
\[
\ch(H^*(\h_{w_0,J}))  = \omega(\csf_q(G,f_J;x,q)).
\]
Hence, for all pairs $w_0,J$ that appear in this way we have a combinatorial description of $\ch(H^*(\h_{w_0,J}))$.
\end{example}

\section{The character of the open cell}
\label{sec:open}
The goal of this section is to prove Theorems \ref{thm:groj_haiman} and \ref{thm:plethysm}. We begin with a proposition.
\begin{proposition}
\label{prop:PJB_bundle}
   Let $J\subset S$ and $w\in W$ be such that $w$ is maximal in $W_JwW_J$. If $w_0$ is the minimum element of $W_Jw$, then $f(\h_w)=\h_{w_0,J}$ and $\h_w\to \h_{w_0,J}$ is a $P_J/B$-bundle.
\end{proposition}
\begin{proof}
  Consider the natural map $f\col G\times G/B \to G\times G/P_J$. We claim that $f^{-1}(f(\h_w))=\h_w$. Indeed, since $\h_w=\{(X,gB); g^{-1}Xg\in \overline{B\dw B}\}$ it is enough to prove that $p\overline{B\dw B}p^{-1}=\overline{B\dw B}$ for every $p\in P_J$. This follows from the fact that $w$ is maximal in $W_JwW_J$. Thus $\h_w$ is a $P_J/B$ fiber bundle over its image.\par

  Now we prove that $f(\h_w)=\h_{w_0,J}$. We know by Proposition \ref{prop:YwYwJ_pushforward} that $f(\h_{w_0})=\h_{w_0,J}$. In particular $\h_{w_0,J}\subset f(\h_w)$, both are irreducible and of the same dimension, hence we must have an equality.
\end{proof}

Let $J$ and $w\in W$ be such that $wJw^{-1}=J$ and $w\in {}^JW$. Given a local system $L$ on $\U_{w,J}$ induced by a representation of $W_J^w$, we have that $(f_{w,J})_!(IC_{\h_{w,J}^\circ}(L))$ is a sum of shifted $1$-character sheaves of $G$, where  $f_{w,J}\col \h_{w,J}^\circ\to G$ is the first projection.  Recall that the $1$-character sheaves of $G$  are in bijection with the irreducible representations of $S_n$.

\begin{theorem}
Let $J$ be a subset of simple transpositions of $w$ and let $w\in {}^JW^J$ be such that $wJw^{-1}=J$. Moreover, let $J'\subset J$ such that $wJ'w^{-1}=J'$ and let $L$ be the local system on $\U_{w,J}$ induced by the induced representation $\ind_{W_{J'}^w}^{W_J^w}$ of $W_J^w$. Then
   \[
   \ch( (f_{w,J})_! (IC_{Y_{w,J}^\circ}(L)[-\ell(w)]))=\frac{\ch(q^{\frac{\ell(w_{J'})}{2}}C'_{w_{J'}}T_w)}{|W_{J'}|},
   \]
   where $w_{J'}$ is the maximum element of $W_{J'}$.
 \end{theorem}
 \begin{proof}
 Consider the map $f\col Y_{w,J'}^{\circ}\to Y_{w, J}^\circ$, we have that $f_*(\mathbb{C}_{Y_{w,J'}^{\circ}})=IC(L)[-\ell(w)]$, then it is enough to prove
 \[
 \ch( (f_{w,J'})_!(\mathbb{C}_{Y_{w,J'}^\circ}))=\frac{q^{\frac{\ell(w_{J'})}{2}}\ch(C'_{w_{J'}}T_w)}{|W_{J'}|}.
 \]
 So we can assume that $J'=J$.

 Since $J'=J$, we have $L=\mathbb{C}_{Y_{w,J}^\circ}$. Consider $h\col G\times G/B\to G\times G/P_J$. By Proposition \ref{prop:PJB_bundle} and the fact that $wJw^{-1}=J$, we have that $h^{-1}(Y_{w,J}^\circ)=\bigsqcup_{z\in W_J} Y_{zw}^{\circ}$. This means that $h_!(\mathbb{C}_{h^{-1}(Y_{w,J}^{\circ})}) = h_*(\mathbb{C}_{h^{-1}(Y_{w,J}^{\circ})})=H^*(P_J/B)\otimes \mathbb{C}_{Y_{w,J}^\circ}$, and hence
 \begin{align*}
 \ch((f_{w,J})_!\mathbb{C}_{Y_{w,J}^\circ})) &= \frac{\ch((f_{w,j}\circ h)_!(\mathbb{C}_{h^{-1}(\h_{w,J}^\circ)}))}{|W_{J}|}\\
                                                & = \frac{\sum_{z\in W_J}\ch((f_{w,J}\circ h)_!(\mathbb{C}_{h^{-1}(\h_{zw,J}^\circ)}))}{|W_{J}|}\\
                                                & = \frac{\sum_{z\in W_J}\ch(T_zT_w)}{|W_{J}|}\\
                                                & = \frac{q^{\ell(w_J)}\ch(C'_{w_J}T_w)}{|W_J|},
 \end{align*}
where the third equality follows from Proposition \ref{prop:ch_lusztig}.
 \end{proof}

 We now prove Theorem \ref{thm:plethysm}.

 \begin{proof}
  Since we are specializing to $q=1$, then $C'_{w_J}(q=1)=\sum_{z\in W_J} T_z(q=1)$. All we have to do is to prove that
  \[
  \frac{\sum_{z\in W_{J}}p_{\lambda(zw)}}{|W_J|_{q=1}} = \prod_{j=1}^mp_{|\tau_j|}[h_{\lambda_{\tau_j}}].
  \]
  However, the cycle decomposition of $zw$ depends only on the cycles of $w$ as a permutation of $\{1,\ldots, \ell(\lambda)\}$, so we can write
  \[
  \sum_{z\in W_{J}}p_{\lambda(zw)} = \prod_{j=1}^m\bigg(\sum_{z_j\in W_{\tau_j}}p_{\lambda(z_j\tau_j)}\bigg).
  \]
  This means that we can restrict to the case when $w$ acts as a cycle on $S_{k}\times\ldots\times  S_{k}$ (that is $w(j)=j+k \pmod{\ell k}$). Thus the identity we want to prove becomes
  \[
 \frac{\sum_{z\in S_k^\ell}p_{\lambda(zw)}}{k!^\ell} =p_\ell[h_{k}].
  \]
  Writing $z=(z_1,\ldots, z_\ell)$, we have that the cycle type $\lambda(zw)$ of $zw$ depends only on the product $z_1z_2\ldots z_\ell$. From this we deduce that $\lambda(zw)=\ell\lambda(z_1z_2\ldots z_\ell)$, thus
  \begin{align*}
  \frac{\sum_{z\in S_k}p_{\ell\lambda(z)}}{k!}
        &=\frac{\sum_{z\in S_k}p_{\ell}[p_{\lambda(z)}]}{k!}\\
        &=p_{\ell}\bigg[\frac{\sum_{z\in S_k}p_{\lambda(z)}}{k!}\bigg]\\
        &= p_{\ell}[h_{k}].
  \end{align*}
 \end{proof}



\section{Projections to splitting spaces of Grassmanians}
\label{sec:J}

In this section we fix $G=GL_n$ and $J=\{n-k+1,\ldots, n-1\}$ for $k=0,\ldots, n$, so that
\[
G/P_J=\{\{0\}\subset V_1\subset \ldots \subset V_{n-k}\subset \mathbb{C}^n\}
\]
is the splitting space of the Grassmanian $\Gr(n-k,n)$. We begin with the following characterization.

\begin{lemma}
\label{lem:wJw=Sn-k}
We have the following equality
\[
\{w\in {}^JS_n; wJw^{-1}=J\} = \begin{cases}
S_{n-k}\times S_1^k & \text{ if } k\geq 2,\\ 
S_n & \text{ otherwise.}
\end{cases}
\]
\end{lemma}
\begin{proof}
If $k=0,1$ we have that $J=\emptyset$ and the result is straightforward. Let us assume that $k\geq 2$. It is clear that $S_{n-k}\times S_1^{k}\subset \{z\in {}^JS_n; wJw^{-1}=J\} $. Let us prove the opposite inclusion. Take $w\in{}^JS_n$ such that $wJw^{-1}=J$. The condition $w\in {}^JS_n$ means that
\[
w^{-1}(n-k+1)<w^{-1}(n-k+2)<\ldots< w^{-1}(n),
\]
while the condition $wJw^{-1}=J$ says that, for every $\ell\in J$, $\{w^{-1}(\ell), w^{-1}(\ell+1)\}=\{\ell',\ell'+1\}$ for some $\ell'\in J$. So we must have that $w^{-1}(\ell)=\ell$ for every $\ell\in J$, which means that $w\in S_{n-k}\times S_1^k$.
\end{proof}
We have the following immediate corollary characterizing the possible subgroups $W_J^w$ appearing in Theorem \ref{thm:local_system_parabolic} when $J=\{n-k+1,\ldots, n-1\}$.

\begin{corollary}
\label{cor:W_J^w_nk1}
Let $w\in {}^JS_n$ such that $wJw^{-1}=J$, then $W_J^w = W_J=S_1^{n-k}\times S_k$.
\end{corollary}

The next lemma proves that for a general $w\in {}^JS_n$, the subset $J_w$ is also of the form $\{n-k'+1,\ldots, n-1\}$ (recall the isomorphism $\h_{w,J}^\circ\to \h_{w,J_w}^\circ$ in Proposition \ref{prop:t_iso_t1}).
\begin{lemma}
Let $w\in {}^JS_n$, then there exists $k'$ such that $J_w=\{n-k'+1,\ldots, n-1\}$.
\end{lemma}
\begin{proof}
Let $\ell_0$ be the minimum element of $J_w$ (recall that $J_w\subset J$). Then we have that $\{w^{-1}(\ell_0), w^{-1}(\ell_0+1)\}=\{\ell'_0,\ell'_0+1\}$ for some $\ell'_0\in J_w$. However, we have that
\[
w^{-1}(\ell_0)<w^{-1}(\ell_0+1)<\ldots < w^{-1} (n),
\]
which means that $\ell'_0\leq \ell_0$ and hence $\ell'_0=\ell_0$ by the minimality of $\ell$. Moreover, we have $w^{-1}(\ell)=\ell$ for every $\ell\geq \ell_0$, which proves that $J_w = \{\ell_0,\ell_0+1,\ldots, n-1\}$.
\end{proof}

Let $w\in {}^JS_n$ such that $wJw^{-1}=J$. Much of the work done in the previous sections becomes easier when $J=\{n-k+1,\ldots, n-1\}$. For instance,  we can actually take $\U_{w,J}$ (the open set appearing in Theorem \ref{thm:UwJ_intro}) to be equal to $(\h_{w,J}^\circ)^{rs}$.


\begin{proposition}
We have a natural map $\pi_1((\h_{w,J}^\circ)^{rs}, (X,gB))\to W_J^w=S_{k}$.
\end{proposition}
\begin{proof}
If $k=0,1$, we have that $S_k$ is trivial so there is nothing to do. Let us assume that $k\geq 2$. By Lemma \ref{lem:wJw=Sn-k} we have that $w\in S_{n-k}\times S_1^{k}$. This means that $r_{n-k,n-k}(w)=n-k$. Recall that
\[
(\h_{e,\{n-k\}^c}^{\circ})^{rs}=\{(X,V_{n-k});X\in G^{rs}, XV_{n-k} = V_{n-k}\}.
\]
We have  a map
\begin{align*}
(\h_{w,J}^\circ)^{rs} &\to (\h_{e,\{n-k\}^c}^{\circ})^{rs}\\
  (X,V_\bullet)&\mapsto (X, V_{n-k}).
\end{align*}
This induces a map $\pi_1((\h_{w,J}^\circ)^{rs}, (X,V_\bullet))\to \pi_1((\h_{e,\{n-k\}^c}^{\circ})^{rs},(X,V_{n-k}))$. \par

We have an $S_{k}$-Galois cover given by
\begin{align*}
    (\h_{e, \{1,\ldots, n-k-1\}}^{\circ})^{rs}&\to (\h_{e,\{n-k\}^c}^{\circ})^{rs}\\
     (X,V_{n-k}\subset V_{n-k+1}\subset \ldots \subset V_{n-1}\subset \mathbb{C}^n)&\mapsto (X,V_{n-k}),
\end{align*}
which gives a map $\pi_1((\h_{e,\{n-k\}^c}^{\circ})^{rs}, (X,V_{n-k}))\to S_{k}$.
\end{proof}

Let $w\in {}^JS_n$ such that $J_w = \{n-k'+1,\ldots, n-1\}$, in particular $w\in S_{n-k'}\times S_1^{k'}$. Let $\overline{w}=w|_{[n-k']}$ and define $\overline{J}=J\cap [n-k']$.

\begin{proposition}
Let $X$ be a regular semisimple matrix. We have that $\h_{w,J}(X)$ is the union of $\binom{n}{n-k'}$ varieties $\h_i$, where each $\h_i$ is isomorphic to $\h_{\overline{w},\overline{J}}(X_i)$ for some regular semisimple $X_i\in GL_{n-k'}$.
\end{proposition}
\begin{proof}
Consider the following commutative diagram
\[
\begin{tikzcd}
  \h_{w, J\setminus\{n-k'\}}^{rs} \ar[r, "f"] \ar[d, "g"]& \h_{w,J}^{rs} \ar[d, "p"]\\
    \h_{e, \{n-k'\}^c}^{rs}\ar[r, "h"] & G^{rs},
   \end{tikzcd}
\]
where
\begin{align*}
    f(X,V_1\subset\ldots \subset V_{n-k}\subset V_{n-k'}\subset \C^n)=&(X,V_1\subset \ldots \subset V_{n-k}\subset \C^n), \\
    g(X,V_1\subset\ldots \subset V_{n-k}\subset V_{n-k'}\subset \C^n)=&(X,V_{n-k'}), \\
    p(X,V_1\subset \ldots \subset V_{n-k}\subset \C^n) =& X,\\
    h(X,V_{n-k'}) = X.
\end{align*}
Note that, since $w\in S_{n-k'}\times S_1^k$, we have that $(X,V_\bullet)\in \h_{w,J\setminus\{n-k'\}}^{rs}$ satisfies $XV_{n-k'}=V_{n-k'}$. The map $f$ is generically injective, since for each $(X,V_\bullet)\in (\h_{w,J}^\circ)^{rs}$ there exists exactly one $V_{n-k'}$ such that $V_{n-k} \subset V_{n-k'}$ and $XV_{n-k'}=V_{n-k'}$. Moreover, the map $f$ is finite, since there is a finite number of $V_{n-k'}$ satisfying $XV_{n-k'}=V_{n-k'}$, because $X$ is regular semisimple. For fixed $X\in G^{rs}$, $|h^{-1}(X)| =\binom{n}{n-k'}$. For fixed $(X,V_{n-k'})\in h^{-1}(X)$, $g^{-1}(X, V_{n-k'})$ is isomorphic to $\h_{\overline{w},\overline{J}}(X|_{V_{n-k'}})$. Moreover $f|_{g^{-1}(X, V_{n-k'})}$ is injective, which means
  \[
  \h_{w,J}(X) = \bigcup_{(X,V_{n-k'})\in h^{-1}(X)} f(g^{-1}(X,V_{n-k'})),
  \]
  which finishes the proof.
\end{proof}

\begin{proposition}
\label{prop:split_local_sytem}
Let $\rho$ be a representation of $S_{k'}$ and consider $L$ the induced character sheaf with support in $\h_{w,J}$. Then
\[
\ch(IH^*(\h_{w,J}(X), IC_{\h_{w,J}}(L)))=\ch(IH^*(\h_{\overline{w},\overline{J}}(X')))\ch(\rho),
\]
where $X\in GL_n^{rs}$ and $X'\in GL_{n-k'}^{rs}$.
\end{proposition}
\begin{proof}
It is sufficient to prove the proposition for $\rho=\ind_{S_{\lambda'}}^{S_{k'}}1$, where $\lambda$ is a partition of $k'$. Fix $\lambda$, consider $\lambda'$ the transpose partition, and define
\[
J_{\lambda} = J\setminus \{n-k', n-k'+\lambda'_1,n-k'+\lambda'_1+\lambda'_2,\ldots n-k'+\lambda'_1+\ldots \lambda'_{\ell(\lambda')}\}.
\]
That means that $(X,V_\bullet)\in \h_{w,J_{\lambda}}$ satisfies $XV_{n-k'+\lambda_1'+\ldots \lambda_j'}=V_{n-k'+\lambda_1'+\ldots \lambda_j'}$. Consider the map
\begin{align*}
f\col \h_{w,J_{\lambda}}^{rs}&\to \h_{w,J}^{rs}\\
(X,V_\bullet)&\mapsto (X, V_1\subset V_2\subset \ldots V_{n-k}).
\end{align*}
We have that $IC_{\h_{w,J}}(L) = f_*(IC_{\h_{w,J_\lambda}})$ because the map $f$ is finite and hence small. On the other hand, we have that $\h_{w,J,\lambda'}(X) = \bigsqcup \h_{\overline{w},\overline{J}}(X|_{V_{n-k'}})$, where the union runs through all $V_{n-k'}\subset V_{n-k'+\lambda'_1} \subset\ldots \subset \C^n$ fixed by $X$. In particular $IH^*(\h_{w,J,\lambda'}(X))=\bigoplus IH^*(\h_{\overline{w},\overline{J}}(X|_{V_{n-k'}}))$. This means that the representation $IH^*(\h_{w,J,\lambda'}(X))$ is induced from the representation $IH(\h_{\overline{w},\overline{J}}(X'))\times \rho$  of $S_{n-k'}\times S_k'$ to $S_n$. The result follows.
\end{proof}

\begin{Exa}
\label{exa:projective_local_system}
Finishing Example \ref{exa:projective}, if $L_i$ is a local system on (a open subset of) $\Hi_i$ induced by the representation $\rho$ of $S_i$, we have
\[
\ch(IH^*(\Hi_i,L_i)) = [n-i]_qh_{n-i}\ch(\rho).
\]
\end{Exa}

This finishes the proof of Theorem \ref{thm:chJnk1} and its corollaries.

\section{Further directions}
\label{sec:further_directions}

We assume that $G=GL_n$ in all the discussions below.

\subsection{Kazhdan--Lusztig polynomials}
We have that $(IC_{\h_w(X)})_p$ for $p\in \h_z(X)^\circ$ has Poincaré polynomial precisely the Kazhdan--Lusztig polynomial $P_{z,w}(q)$. In other words, the singularity of $\h_w$ at $\h_z^\circ$ is the same as the singularity of the Schubert variety $\Omega_w$ at $\Omega_z^\circ$. \par

The picture for parabolic Lusztig varieties is not so clear. For instance, $\Omega_{2134,\{2,3\}}$ is smooth, while $\h_{2134,\{2,3\}}$ is singular at $\h_{1234,\{2,3\}}$. In fact, for an invertible regular semisimple matrix $X$, we have that $\h_{2134,\{2,3\}}(X)\subset \mathbb{P}^{3}$ is the union of the $6$ coordinate lines, while $\h_{1234,\{2,3\}}(X)$ is the set of the coordinate points of $\mathbb{P}^3$. So $(IC_{\h_{2134,\{2,3\}}})|_{\h_{1234,\{2,3\}}}$ is a local system on $\h_{1234,\{2,3\}}$ of rank $3$. More than that, it is associated to the representation of $S_3$ whose character is $h_{2,1}$ (a proper proof of this statement will be given in a subsequent work).

It is not clear to us how the parabolic Kazhdan--Lusztig polynomials (introduced by Deodhar in \cite{DeodharParabolic}) relate to the singularities of the parabolic Lusztig varieties, or if new polynomials (or symmetric functions/graded representations) will have to be defined.


\subsection{Positivity}
The natural extension of the Stanley--Stembridge and Haiman's conjecture to the parabolic case is:

\begin{conjecture}
The Frobenius character $\ch(IH^*(\h_{w,J}(X))$ is $h$-positive for every $w \in S_n$ and for every $J\subset \{1,\ldots, n-1\}$.
\end{conjecture}
It is important to notice that $\ch(IH^*(\h_{w,J}(X),L))$ may fail to be $h$-positive (see Example \ref{exa:ch_G24}) even if $L$ is induced by a permutation representation of $W_{J_w}^w$. We could ask

\begin{Question}
Given $w$ and $J$, which local systems $L$ satisfies that $\ch(IH^*(\h_{w,J}(X),L))$ is $h$-positive?
\end{Question}

\subsection{Combinatorial and algebraic interpretations}

We know that (\cite{ChaShvV, BrosnanChow, AN_hecke})
\begin{align*}
\ch(IH^*(\h_w(X)))=&\ch(q^{\frac{\ell(w)}{2}}C'_w)\\
\ch(H^*(\h_{w_\m}(X)))=&\omega(\csf_q(G_\m))
\end{align*}
for every permutation $w\in S_n$ and for every Hessenberg function $\m\col[n]\to[n]$. What are the analogues of these equalities in the parabolic case?

\begin{Question}
\label{ques:codo}
Give a combinatorial/algebraic interpretation of  $\ch(IH^*(\h_{w,J}(X)))$ for some class of pairs $(w, J)$.
\end{Question}

Example \ref{exa:chromatic_gasharov} (see also \cite{KiemLee}) gives a class of pairs $(w_0,J)$ for which we can give such a combinatorial description. Unfortunately, this class seems to be very restrictive. Some classes the authors consider of special importance are
 \begin{itemize}
     \item when $J=\{n-k+1,\ldots, n-1\}$ and $w$ is codominant,
     \item when $J=\{k\}^c$, that is $G/P_J$ is a Grassmannian.
 \end{itemize}

The authors consider the first class above as an important class, because examples indicate that:

\begin{conjecture}
\label{conj:local_system_codominant}
Let $w$ be a codominant permutation, $J=\{n-k+1,\ldots, n-1\}$ and $f\col G\times \flag \to G\times \flag_J$ be the forgetful map. If
\[
f_*(IC_{\h_w})=\bigoplus_{z\in {}^JW} IC_{\h_{z,J}}(L_{z,w}^{J',\emptyset})
\]
is the splliting given by the decomposition theorem, then $L_{z,w}^{J',\emptyset} = 0$ for every non-codominant permutation $z$.
\end{conjecture}

We could also ask for combinatorial interpretations of the characters of the local systems that appear in the decomposition theorem. Since the groups $W_{J_z}^z$ can be complicated in general, we restrict ourselves to the case $J=\{n-k+1,\ldots, n-1\}$ (which means that $W_{J_z}^z = S_{i}$ for some $i$ that depends on $z$ and $J$).

\begin{Question}
\label{ques:L_codo}
Let $J=\{n-k+1,\ldots, n-1\}$ and $J'=\{n-k'+1,\ldots, n-1\}$ with $k'>k$. Give a combinatorial description of $\ch(L_{w, z}^{J,J'})$, when $w, z$ are codominant permutations. Recall that $L_{w,z}^{J,J'}$ is induced by a representation of $S_i$ for some $i$ that depends on $z$ and $J'$.
\end{Question}

Assuming that Conjecture \ref{conj:local_system_codominant} is true, we have that Theorem \ref{thm:local_system_parabolic} and Proposition \ref{prop:split_local_sytem} give
\[
\ch(IH^*(\h_{w,J}(X))) = \sum_{\substack{z\leq w, z\in {}^{J'}S_n\\ z \text{ codominant}}} \ch(\h_{\overline{z},\overline{J'}}(X_z))\ch(L_{w, z}^{J,J'}).
\]
This gives a recursion for computing $\ch(IH^*(\h_{w,J}(X)))$ (assuming that $\ch(L_{w, z}^{J,J'})$ is known), where the initial case is Example \ref{exa:projective}. Even the case $k'=k+1$ is already relevant. Also noteworthy is the fact that $h$-positivity of $\ch(L_{w, z}^{J,J'})$ for every $w, z$ codominant permutations and for every $J=\{n-k+1,\ldots, n-1\}$, $J'=\{n-(k+1)+1,\ldots, n-1\}$, would imply the $e$-positivity $\csf_q(G_\m)$ (also assuming the validity of Conjecture \ref{conj:local_system_codominant}).

\subsection{LLT polynomials}

By \cite{Procesi}, \cite{GP} and \cite[Proposition 5.4]{PrecupSommers} there exists a representation $\LLT_\m$ for each Hessenberg function $\m$ such that
\begin{equation}
    \label{eq:LLT}
    P(G/B)\otimes \LLT_\m = \mathcal{C}\otimes H^*(\h_{w_\m}(X)),
\end{equation}
where $P(G/B)$ is the trivial representation of $S_n$ in $H^*(G/B)$ and $\mathcal{C}$ is the coinvariant algebra of $S_n$. The Frobenius characters of $\LLT_\m$ are precisely the unicellular LLT-polynomials (see \cite{LLT}, \cite{CarlssonMellit}, \cite{AlexPanova}). \par

Is there an analogue of Equation \eqref{eq:LLT} for the parabolic Lusztig varieties? That is, for a fixed $J\subset\{1,\ldots, n-1\}$ there exists representations $P_J$ and $\mathcal{C_J}$ such that, for each $w\in {}^JS_n$, there exists a representation $\LLT_{w,J}$ satisfying
\[
P_J\otimes \LLT_{w,J} = \mathcal{C}_J\otimes IH^*(\h_{w,J}(X)).
\]



\bibliographystyle{amsalpha}
\bibliography{bibli}

\end{document}